\documentclass[11pt,a4paper]{article}
\title{Equitable Partitions into \\
Matchings and Coverings in Mixed Graphs}
\author{
Tam\'as Kir\'aly
\thanks{MTA-ELTE Egerv\'ary Research Group, E\"otv\"os University, Budapest. E-mail: {\tt tkiraly@cs.elte.hu}}
\and
Yu Yokoi
\thanks{National Institute of Informatics, Tokyo, Japan. E-mail: {\tt yokoi@nii.ac.jp}
}}
\date{}

\usepackage{latexsym,amssymb,amsmath,amsthm,braket,comment}
\usepackage{color,calc,bm, mathrsfs, setspace, graphics}
\usepackage[margin=2.7cm]{geometry}

\usepackage{soul,xcolor}
\setstcolor{red}
\renewcommand{\st}[1]{}

\usepackage{graphicx}
\newtheorem{theorem}{Theorem}[section]
\newtheorem{lemma}[theorem]{Lemma}

\newtheorem{proposition}[theorem]{Proposition}
\newtheorem{corollary}[theorem]{Corollary}

\newtheorem{claim}[theorem]{Claim}

\newcommand{\Z}{\mathbf{Z}}
\newcommand{\R}{\mathbf{R}}
\newcommand{\cS}{\mathcal{S}}
\newcommand{\mixsize}[1]{|{#1}|_{\rm mix}}
\newcommand{\dist}{\text{\rm dist}_{G}}
\newcommand{\Dist}{D}
\newcommand{\covered}{\partial}
\begin{document}

\maketitle
%\todo{ Tam\'as: I added corollaries on bibranchings at the end. I also added that our constructions are polynomial time. I hope it is true :) }

\begin{abstract}
Matchings and coverings are central topics in graph theory.
The close relationship between these two has been key to
many fundamental algorithmic and polyhedral results.
For mixed graphs, the notion of matching forest was proposed as a common generalization of matchings and branchings.

In this paper, we propose the notion of mixed edge cover as a covering counterpart of matching forest,
and extend the matching--covering framework to mixed graphs.
While algorithmic and polyhedral results extend fairly easily,
partition problems are considerably more difficult in the mixed case.
We address the problems of partitioning a mixed graph into matching forests or mixed edge covers,
so that all parts are equal with respect to some criterion, such as edge/arc numbers or total sizes.
Moreover, we provide the best possible multicriteria equalization.
%\keywords{Matching  \and Edge cover \and Mixed graph}
\end{abstract}
\section{Introduction}
Let $G=(V, E\cup A)$ be a mixed graph with undirected edges $E$ and directed arcs $A$.
In this paper, we use the term `edge' only for undirected edges.
Graphs have no loops (edge/arc), but may have parallel edges or arcs.
Each arc has one head and we regard both endpoints of an edge as heads.
We say that $v\in V$ is {\bf covered by} an edge/arc $e\in E\cup A$ if $v$ is a head of $e$.
A {\bf matching forest}, introduced by Giles \cite{G82a,G82b,G82c}, is a subset $F\subseteq E\cup A$ such that
(i) the underlying undirected graph has no cycle and (ii) every vertex $v\in V$ is covered at most once in $F$.
This is a common generalization of the notion of matching in undirected graphs and the notion of branching in directed graphs.
A matching forest is {\bf perfect} if it covers every vertex exactly once, i.e., every $v\in V$ is the head of exactly one edge or arc.
Matching forests have been studied in order to unify fundamental theorems about matchings and branchings. In particular, unifying results were given on total dual integrality by  Schrijver \cite{Schr00}, on Vizing-type theorems by Keijsper \cite{K03}, and on the delta-matroid property of degree-sequences by Takazawa \cite{T14}. \textcolor{black}{One of the main contributions of the present paper is a new unifying result concerning equitable partitions. We consider the problem of partitioning into matching forests of almost equal edge-size and arc-size, which generalizes the well-known equitable partition properties of matchings and branchings.}

\subsection*{Mixed Edge Covers}

In undirected graphs, as shown by Gallai's theorem \cite{Gallai} and other results,
matching is closely related to edge cover, a set of edges covering all vertices.
Another contribution of our paper is to introduce and analyze a covering counterpart of the notion of matching forest, that can be regarded as a common generalization of edge covers and bibranchings.
We present two natural ways to define covering structures in mixed graphs; later we will show that these two are in some sense equivalent.
%, which will be shown to be related to matching forests.
First, we may relax the requirements in undirected edge cover: instead of requiring each vertex to be covered by an edge, we only require each vertex to be reachable from an edge.
This results in the following version of edge cover for mixed graphs.
\begin{itemize}
\setlength{\leftskip}{-2mm}
\item A {\bf mixed edge cover} in a mixed graph $G=(V, E\cup A)$ is a subset $F\subseteq E\cup A$ such that
for any $v\in V$, there is a directed path (which can be of length $0$)
in $F\cap A$ from an endpoint of some $e\in F\cap E$ to $v$.
\end{itemize}
If the graph is undirected, then this notion coincides with edge cover.
Also, bibranchings in a partitionable directed graph
%\footnote{A directed graph is called partitionable if its vertex set $V$ can be partitioned into $V_{1}$ and $V_{2}$ such that there is no arc from $V_{2}$ to $V_{1}$.}
can be represented as mixed edge covers in an associated mixed graph (see Section \ref{sec:mcf}).
Thus, mixed edge cover generalizes both edge cover and bibranching.
Alternatively, the following notion may also be considered as a covering counterpart of matching forest.
\begin{itemize}
\setlength{\leftskip}{-2mm}
\item A {\bf mixed covering forest} in a mixed graph $G=(V, E\cup A)$ is a subset $F\subseteq E\cup A$
such that (i) the underlying undirected graph has no cycle and (ii) every vertex $v\in V$ is covered at least once in $F$.
\end{itemize}
These two notions coincide if (inclusionwise) minimality is assumed.
That is, a  minimal mixed edge cover is also a minimal mixed covering forest and vice versa (see Proposition~\ref{prop:equivalence}).
In case of nonnegative weight minimization or packing problems,
where the optimal solutions can be assumed to be minimal, the terms are interchangeable. This is however not true for partitioning problems.
In this paper we mainly work with mixed edge covers, and obtain results on mixed covering forests as consequences.

\st{Our results can be divided into the following two parts.
In the first part we show that results on matching and edge-cover naturally extend to mixed graphs,
while the second part deals with new problems which arise from the heterogeneous nature of mixed graphs.}

\textcolor{black}{Before proving our main results on equitable partitioning, we first show some
connections between matching forest and mixed edge cover in Sections~\ref{sec:preliminary} and \ref{sec:algpol}.}
In undirected graphs, matching and edge cover are closely related,
and for both of them, polyhedral and algorithmic results are known.
For mixed graphs, however, only matching forests have been investigated.
\textcolor{black}{We provide several results which show that mixed edge covers exhibit similar properties in mixed graphs as edge covers do in undirected graphs.}
\st{connections between matching forest and mixed edge cover, and use these connections to derive polyhedral and algorithmic results on mixed edge covers.}
\st{We first generalize Gallai's theorem, whose original statement is as follows:...}
%\emph{For any undirected graph without isolated vertices, the sum of the cardinalities of a maximum matching and a minimum edge cover is $|V|$.}
\st{For a mixed graph, define the mix-size...}
%$\mixsize{F}$ of a subset $F\subseteq E\cup A$ by
%$\mixsize{F}:=|F\cap E|+\frac{1}{2}|F\cap A|$.
\st{With this mix-size, the statement of Gallai's theorem holds for matching forests and mixed edge covers.}

\textcolor{black}{We first show that Gallai's theorem \cite{Gallai} on the sizes of maximum matching and minimum edge cover naturally extends to
matching forests and mixed edge covers (Theorem \ref{thm:Gallai}).}
We then reduce the optimization problem on mixed edge covers to optimization
on perfect matching forests in an auxiliary graph.
This fact immediately implies a polynomial time algorithm to find a minimum weight mixed edge cover.
Furthermore, using this relation we can provide a polyhedral description of the mixed edge cover polytope and show its total dual integrality, obtaining a covering counterpart of the result of Schrijver \cite{Schr00}.
\st{These algorithmic and structural results show that mixed edge covers exhibit similar properties in mixed graphs as edge covers do in undirected graphs.}
\textcolor{black}{In this way, algorithmic and polyhedral aspects of the matching--covering framework of undirected graphs naturally extends to mixed graphs.}

\subsection*{Equitable Partitions in Mixed Graphs}
\textcolor{black}{In contrast to the above results, equitable partition problems in mixed graphs have an additional difficulty stemming from their mixed structure.}
Recall that the notion of matching forest is a common generalization of matchings in undirected graphs and
branchings in directed graphs. These structures are known to have the following \emph{equitable partition property} \cite{SchrBook}:
if the edge set $E$ of an undirected graph (resp., the arc set $A$ of a directed graph) can be partitioned into $k$ matchings (resp., branchings) $F_1, F_2,\dots,F_k$,
then we can re-partition $E$ (resp., $A$) into $k$ matchings (resp., branchings) $F'_1, F'_2,\dots, F'_k$ such that $|F'_i|-|F'_j|\leq 1$ for any $i,j\in [k]$ (where $[k]=\{1,2,\dots,k\}$).
Note that bounding the difference of cardinality by $1$  is the best possible equalization.

In this paper, we consider equitable partitioning into matching forests and into mixed edge covers.
\textcolor{black}{Since the definitions of matching forest and mixed edge cover include conditions that depend on both
the edge part and the arc part, the problem cannot be simply decomposed into two separate problems on edges and arcs. Indeed, achieving a difference bounded by 1 in the arc
part is impossible in general, so a more refined approach is needed.}

Equitable partition problems have been studied extensively for various combinatorial structures, the most famous being the equitable coloring theorem
of Hajnal and Szemer\'edi \cite{HS70} and the stronger conjecture of Meyer \cite{Meyer73}, which is still open.
The equitable partition property of matchings implies that the equitable chromatic number of any line graph equals its chromatic number.
Edge/arc partitioning problems with equality or other cardinality constraints have also been studied for other graph structures \cite{CHR91,FSz11,Werra03,Werra05}.

\paragraph{Partitioning into Matching Forests. }
%\textcolor{black}{The equitable partition property of matchings implies that the equitable chromatic number \cite{HS70,Meyer73} of any line graph coincides with its chromatic number. %Also, edge/arc partitioning problems with equality or other cardinality constraints have been studied for various structures \cite{CHR91,FSz11,Werra03,Werra05}.
%\todo{ Perhaps not enough, modify freely. Added papers are in the folder "references".}}
%(This is clear if $|E|/k$ is fractional. Also, if $|E|/k$ is integral, this bound automatically
%implies $|F'_i|=|F'_j|$ for all $i,j$ otherwise there is a pair with difference at least $2$).
%Note that it is NP-complete to decide whether a given edge set $E$ can be partitioned into $k$ matchings, since this corresponds to edge-coloring.
Since mixed graphs have two different types of edges, there are several possible criteria for equalization:
the number of edges, the number of arcs, and the total cardinality.
(We call these edge-size, arc-size, and total size, respectively.)
We study equalization with respect to each of these criteria,
as well as the possibility of ``multicriteria equalization.''
%For example, given a matching and branching, we return two matching forests in which edges and arcs are well blended.

It turns out that the coexistence of edges and arcs makes equalization more difficult.
See the graph in Fig.~\ref{fig1}, which consists of two edges and two arcs.
In order to equalize with respect to all of the above criteria, we would need a partition into a pair of matching forests with one edge and one arc in each, but no such partition exists.

\begin{figure}[htb]
\begin{center}
   \includegraphics[width=65mm]{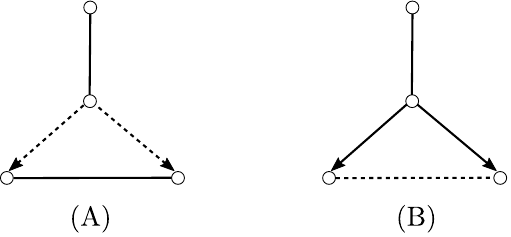}
\caption{There are two possible partitions into two matching forests.
In (A), total size is equalized while in (B) edge-size is equalized.}
\label{fig1}
\end{center}
\end{figure}

In this example, the two arcs are in the same part in any partition into two matching forests.
Thus, unlike in the case of branchings, the difference of $2$ in arc-size is unavoidable in some instances.
The example also shows the impossibility of equalizing edge-size and total size simultaneously.

We show that equalization is possible separately for edge-size and total size.
Also, simultaneous equalization is possible by relaxing one criterion just by $1$.
These results are summarized in the following two theorems.
We sometimes identify a mixed graph $G=(V, E \cup A)$ with $E\cup A$ (e.g., we say ``$G$ is partitionable'' to mean ``$E\cup A$ is partitionable.'')
For a set of matching forests $F_1,\dots,F_k$, we write $M_i:= F_i\cap E$ for their edge parts
and $B_i:= F_i\cap A$ for their arc parts.
\begin{theorem}\label{thm:mf1}
  Let $G=(V, E \cup A)$ be a mixed graph that can be partitioned into $k$ matching forests. Then $G$ can be partitioned into $k$ matching forests $F_1,\dots,F_k$ in such a way that, for every $i,j \in [k]$, we have $||F_i|-|F_j|| \leq 1$,  $||M_i|-|M_j|| \leq 2$, and $||B_i|-|B_j|| \leq 2$.
\end{theorem}
\begin{theorem}\label{thm:mf2}
  Let $G=(V, E \cup A)$ be a mixed graph that can be partitioned into $k$ matching forests. Then $G$ can be partitioned into $k$ matching forests $F_1,\dots,F_k$ in such a way that,
for every $i,j \in [k]$, we have $||F_i|-|F_j|| \leq 2$, $||M_i|-|M_j|| \leq 1$, and $||B_i|-|B_j|| \leq 2$.
\end{theorem}
We remark again that, even if we consider a single criterion,
the minimum differences in $|F_i|$, $|M_i|$, $|B_i|$ can be $1,1,2$ respectively.
These theorems say that relaxing one criterion just by $1$ is sufficient for simultaneous equalization.
%We also remark that in these theorems the difference in mix-size $\mixsize{F_i}$ is bounded by $1$
%because $||F_i|-|F_j||+||M_i|-|M_j|| \leq 2$ and $\mixsize{F_i}=\frac{1}{2}(|F_i|+|M_i|)$.
%\todo{ I commented out the remark for mix-size. We can add if it is necessary.}

\paragraph{Partitioning into Mixed Edge Covers.}
Next, we consider equitable partitioning into mixed edge covers.
In contrast to the first part, where polyhedral and algorithmic results on mixed edge covers are obtained via reduction to matching forests,
there seems to be no easy way to adapt these reductions to equalization problems.
The reason is that the correspondence between matching forest and mixed edge cover presumes
%holds only in their optimum (i.e., under inclusionwise
maximality/minimality, but these cannot be assumed in equitable partitioning problems.
%So we develop new operations for mixed edge covers.

That said, equalization faces similar difficulties as in the case of matching forests.
See the graph in Fig.~\ref{fig2}, which has two components.
Each component has a unique partition into two mixed edge covers,
so the whole graph has only two possible partitions
(one is shown in Fig.~\ref{fig2}, while the other is obtained by flipping the colors in one component.)

\begin{figure}[htb]
\begin{center}
   \includegraphics[width=69mm]{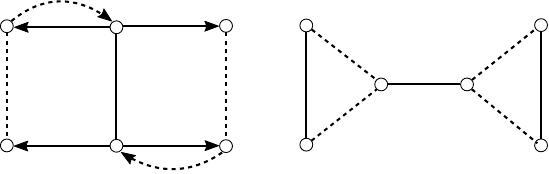}
\caption{A graph that consists of two components. For each component, the partition is unique,
and hence there are two possible partitions for the whole graph.}
\label{fig2}
\end{center}
\end{figure}

This example shows that the difference of $2$ in arc-size is unavoidable, and
simultaneous equalization of edge-size and total size is impossible.
Fortunately, this is the worst case.
Similarly to matching forests, we can obtain the following theorems for mixed edge covers.
For mixed edge covers $F_1,\dots,F_k$, we use the notation $N_i:= F_i\cap E$ for their edge parts
and $B_i:= F_i\cap A$ for their arc parts (the reason for using $N_i$ instead of $M_i$ is to emphasize that $F_i\cap E$ is not necessarily a matching.)

\begin{theorem}\label{thm:mec1}
Let $G=(V, E \cup A)$ be a mixed graph that can be partitioned into $k$ mixed edge covers. Then $G$ can be partitioned into $k$ mixed edge covers $F_1,\dots,F_k$
in such a way that, for every $i,j \in [k]$, we have $||F_i|-|F_j|| \leq 1$, $||N_i|-|N_j|| \leq 2$, and  $||B_i|-|B_j|| \leq 2$.
\end{theorem}

\begin{theorem}\label{thm:mec2}
Let $G=(V, E \cup A)$ be a mixed graph that can be partitioned into $k$ mixed edge covers. Then $G$ can be partitioned into $k$ mixed edge covers $F_1,\dots,F_k$
in such a way that,  for every $i,j \in [k]$, we have  $||F_i|-|F_j|| \leq 2$, $||N_i|-|N_j|| \leq 1$, and $||B_i|-|B_j|| \leq 2$.
\end{theorem}

We now mention equitable partitioning into mixed covering forests, the other type of structure we introduced as a covering counterpart of matching forests.
Unlike mixed edge covers, mixed covering forests require acyclicity, which makes partitioning even harder.
The graph in Fig.~\ref{fig3} has a unique partition into two mixed covering forests, where edge-size is not equalized.
However, if we consider packing rather than partitioning,
then we can show that the corresponding versions of Theorems~\ref{thm:mec1} and \ref{thm:mec2} hold for mixed covering forests.
The formal statements are given in Section~\ref{sec:mcf} as Corollaries~\ref{cor:mcf1} and \ref{cor:mcf2}.

\begin{figure}[htb]
\begin{center}
   \includegraphics[width=50mm]{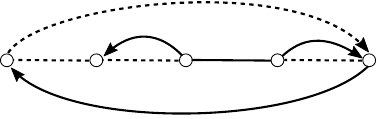}
\caption{A graph that has a unique partition into two mixed covering forests. }
\label{fig3}
\end{center}
\end{figure}

We add two more remarks about the results. First, our multicriteria equalization result is new even for bibranchings. We describe the consequences for bibranchings in Section~\ref{sec:mcf}.

Second, our results are constructive in the sense that if an initial partition $F_1,\dots,F_k$ is given, then our proof gives rise to a polynomial-time algorithm to obtain the desired partition $F_1',\dots,F_k'$ in Theorems~\ref{thm:mf1}, \ref{thm:mf2}, \ref{thm:mec1}, and \ref{thm:mec2}. Note however that it is NP-complete to decide whether a mixed graph can be partitioned into $k$ matching forests or $k$ mixed edge covers, even in the undirected case.

The rest of the paper is organized as follows. Section \ref{sec:preliminary} describes basic properties of matching forests and mixed edge covers, including a new extension of Gallai's theorem. In Section \ref{sec:algpol}, we show that a minimum weight mixed edge cover can be found in polynomial time, and we give a TDI description of the mixed edge cover polytope.  Sections \ref{sec:equitable-mf} and \ref{sec:equitable-mec} contain our results on equitable partitioning of matching forests and mixed edge covers, respectively. In the last subsection, we describe the corollaries for mixed covering forests and bibranchings.

\section{Matching Forests and Mixed Edge Covers}\label{sec:preliminary}
%\subsection{Matching Forests and Mixed Edge Covers}
We describe some basic properties of matching forests and mixed edge covers.
Let $G=(V, E\cup A)$ be a mixed graph. For a subset $F\subseteq E\cup A$,
we say that $v\in V$ is {\bf covered} in $F$ if $v$ is an endpoint of some edge $e\in F$
or is the head of some arc $a\in F$.
We denote by $\covered(F)$ the set of vertices covered in $F$.

An edge set $M\subseteq E$ is a {\bf matching} (resp., {\bf edge cover}) if each vertex is covered at most once (resp., at least once) in $M$.
An arc set $B\subseteq A$ is a {\bf branching} if each vertex is covered at most once in $B$
and there is no directed cycle in $B$.
For a branching $B$, we call $R(B):=V\setminus \covered(B)$  the {\bf root set} of $B$.
Note that, in a branching $B$, any vertex is reachable from some root $r\in R(B)$
via a unique directed path (which can be of length $0$).

We provide characterizations of matching forests and mixed edge covers, where the first one is clear from the definition.
\begin{proposition}\label{prop:mf-chara}
A subset $F\subseteq E\cup A$ is a matching forest if and only if $F\cap A$ is a branching and
$F\cap E$ is a matching such that  $\covered(F\cap E)\subseteq R(F\cap A)$.
\end{proposition}
\begin{proposition}\label{prop:mec-chara}
A subset $F\subseteq E\cup A$ is a mixed edge cover if and only if $F\cap A$ contains a branching $B$
such that $R(B)\subseteq \covered(F\cap E)$.
\end{proposition}
\begin{proof}
The ``if'' part is clear because every $v\in R(B)$ is covered by an edge and every $v\in V\setminus R(B)$ is
reachable from $R(B)$ in $B$. For the ``only if'' part, suppose that $F$ is a mixed edge cover.
By definition, for any $v\in V\setminus \covered(F\cap E)$, there is a directed path from $\covered(F\cap E)$ to $v$.
This means that, if we contract $\covered(F\cap E)$ to a new vertex $r$, then there exists an $r$-arborescence.
In the original graph, this arborescence corresponds to a branching $B$ such that $\covered(B)\supseteq V\setminus \covered(F\cap E)$,
and hence $R(B)=V\setminus \covered(B)\subseteq \covered(F\cap E)$.
\end{proof}

As mentioned in the Introduction, mixed edge covers and mixed covering forests have the following relationship.
\begin{proposition}\label{prop:equivalence}
Every mixed covering forest is a mixed edge cover.
Moreover, a subset $F\subseteq E\cup A$ is a minimal  mixed edge cover if and only if it is a minimal mixed covering forest.
\end{proposition}
\begin{proof}
For the first claim, suppose for contradiction that a mixed covering forest $F$ is not a mixed edge cover.
Then, some vertex $v$ is unreachable from $F\cap E$. Let $U$ be the set of vertices from which $v$ is reachable;
then no $u\in U$ is incident to edges. As $F$ is a covering forest,
every $u\in U$ is covered by some arc $a\in F\cap A$, whose tail is also in $U$ by the definition of $U$.
Therefore, there are at least $|U|$ arcs whose head and tail are both in $U$, which contradicts the acyclicity of $F$.

For the ``if'' part of the second claim, take a minimal mixed covering forest $F$.
This is a mixed edge cover as just shown.
The minimality of $F$ implies that any proper subset of $F$ has some uncovered vertex,
and hence is not a mixed edge cover. So $F$ is a minimal edge cover.

For the ``only if'' part, let $F$ be a minimal mixed edge cover.
By the first claim, it suffices to show that this is a mixed covering forest.
Clearly, all vertices are covered at least once because they are reachable from $\covered(F\cap E)$,
so we have to show acyclicity. Observe that the minimality of $F$ implies that the head $v\in V$
of any arc $a\in F\cap A$ is covered only by $a$ (otherwise we can remove $a$ or another arc whose head is $v$).
Suppose, to the contrary, that $C\subseteq F$ is a cycle in the underlying graph.
If all elements of $C$ are edges, then we can remove at least one edge, which contradicts minimality.
Therefore, $C$ contains some arc $a$. By the above observation, the head $v$ of $a$ is covered only by $a$,
so the other element in $C$ incident to $v$ should be an arc whose tail is $v$.
By repeating this argument, we see that all elements of $C$ are arcs. Then all vertices in $C$ are only covered by arcs in $C$, which means that they are unreachable from $\covered(F\cap E)$, a contradiction.
\end{proof}

%Let us observe a close relation between matching forests and mixed edge covers.
Let us define the mix-size $\mixsize{F}$ of any $F\subseteq E\cup A$ by $\mixsize{F}:=|F\cap E|+\frac{1}{2}|F\cap A|$.
Using this mix-size, we can generalize Gallai's well known  theorem on the relation between maximum matching and minimum edge cover to mixed graphs.
\begin{theorem}\label{thm:Gallai}
For a mixed graph $G=(V, E \cup A)$ that admits a mixed edge cover, let
$\nu(G) :=\max\{~\mixsize{F}: \text{$F$ is a matching forest in $G$}~\}$ and
$\rho(G):=\min\{~\mixsize{H}:$ $H$ is a mixed edge cover in $G~\}$.
Then we have $\nu(G)+\rho(G)=|V|$.
\end{theorem}
\begin{proof}
For any vertex $v$, we denote by $\dist(v)$ the minimum length of a directed path from $\covered(E)$ to $v$.
If $G$ admits a mixed edge cover, then $\dist(v)$ is finite for every $v\in V$.
For any $v\in V$, we have $\dist(v)=0$ if and only if $v\in \covered(E)$.

\begin{claim}
Among matching forests satisfying $\mixsize{F^*}=\nu(G)$, let $F^*$ minimize
\[\Dist(F):=\sum\set{\dist(v)|v\in V\setminus \covered(F)}.\]
Then $\Dist(F^*)=0$, and hence $V\setminus \covered(F^*)\subseteq \covered(E)$.
\end{claim}
Suppose, to the contrary, $\Dist(F^*)>0$,
i.e.\ $\dist(v)\geq 1$ for some $v\in V\setminus \covered(F^*)$.
Take a shortest directed path $P$ from $\covered(E)$ to $v$ and let $a\in P$ be the arc whose head is $v$.
Since $v$ is uncovered in $F^{*}$, every vertex is covered at most once in $F^{*}+a$,
which is not a matching forest by the maximality of $F^{*}$.
This means that there exists a directed cycle $C$ with $a\in C\subseteq F^{*}+a$.
Let $a'\in C$ be the arc preceding $a$ in $C$ and let $u$ be the head of $a'$ (which is also the tail of $a$).
Then $F':=F^{*}+a-a'$ is a matching forest and satisfies $\mixsize{F'}=\mixsize{F^*}=\nu(G)$.
Because $\covered(F')=\covered(F^{*})-u+v$, we have $\Dist(F')=\Dist(F^*)+\dist(u)-\dist(v)$.
As $u$ is on the shortest path to $v$, we see $\dist(u)\leq \dist(v)-1$, and hence
$\Dist(F')<\Dist(F^*)$, which contradicts the choice of $F^*$. The claim is proved.

By this claim, every $v\in V\setminus\covered(F^*)$ is incident to some edge.
\begin{claim}
$\rho(G)\leq |V|-\nu(G)$.
\end{claim}
Let $H$ be a superset of $F^*$ obtained by adding an arbitrary incident edge for each $v\in V\setminus\covered(F^*)$.
Then $H$ is a mixed edge cover. To see this,
we show that any $v\in V$ is reachable from $\covered(H\cap E)$ in $F^*\cap A$.
By Proposition~\ref{prop:mf-chara}, $F^*\cap A$ forms a branching.
Let $r\in V$ be the root of the component containing $v$ (which can be $v$ itself).
Then $v$ is reachable from $r$ in $F^*\cap A$. Because $r$ is not covered by any arc,
$r\in \covered(F^*\cap E)$ or $r\in V\setminus \covered(F^*)$.
Both of them imply $r\in \covered(H\cap E)$ by the definition of $H$,
and hence $v$ is reachable from $\covered(H\cap E)$.
Thus, $H$ is a mixed edge cover, and we have $\mixsize{H}\geq \rho(G)$.

Because $F^*$ has $2\mixsize{F^*}$ heads,  we have $|V\setminus\covered(F^*)|=|V|-2\mixsize{F^*}$,
and $\mixsize{H}=\mixsize{F^*}+(|V|-2\mixsize{F^*})$ by the construction of $H$.
Hence, we obtain $\rho(G)\leq \mixsize{H}=|V|-\mixsize{F^*}=|V|-\nu(G)$.

\begin{claim}
$\rho(G)\geq |V|-\nu(G)$.
\end{claim}
Take a mixed edge cover with $\mixsize{H^*}=\rho(G)$
and let $F$ be an inclusion-wise maximal matching forest in $H^*$.
By the minimality of $H^*$, the head of any arc $a\in H^*\cap A$ is covered only by $a$ in $H^*$.
Also, Proposition~\ref{prop:equivalence} implies that the underlying graph of $H^*$ has no cycle.
Thus $F$ includes $H^*\cap A$, and hence $H^*\setminus F \subseteq E$ and $\mixsize{H^*}-\mixsize{F}=|H^*\setminus F|$.
By the maximality of $F$, any edge $e\in H^*\setminus F$ has at most one endpoint in $V\setminus \covered(F)$,
while $V\setminus \covered(F)\subseteq \covered(H^*\setminus F)$.
Then, $|V\setminus \covered(F)|\leq |H^*\setminus F|$, which implies
$|V|-2\mixsize{F}=|V\setminus \covered(F)|\leq |H^*\setminus F|=\mixsize{H^*}-\mixsize{F}$.
Thus, $\rho(G)=\mixsize{H^*}\geq |V|-\mixsize{F}\geq |V|-\nu(G)$.
\end{proof}

\section{Algorithms and Polyhedral Descriptions}\label{sec:algpol}
\subsection{Previous Results on Matching Forests}
We introduce some known results on matching forests that will be used in our proofs for mixed edge covers in Section~\ref{sec:mic-property}. Giles \cite{G82b} showed that the maximum weight matching forest problem is tractable.

\begin{theorem}[Giles \cite{G82b}]\label{thm:mfalgo}
There is a strongly polynomial-time algorithm to find a maximum weight matching forest or a maximum weight perfect matching forest, for any weight function $w:E \cup A \to \R$.
\end{theorem}

Giles also gave a linear description of the matching forest polytope and characterized its facets \cite{G82b,G82c}. It was later shown by Schrijver that this system is totally dual integral (TDI). To state the result, we call a collection of subpartitions $\cS_1,\cS_2,\dots,\cS_k$ \emph{laminar} if for any $i$ and $j$, one of the following is true:
\begin{itemize}
  \item for every $X \in \cS_i$, there exists $Y \in \cS_j$ such that $X \subseteq Y$,
  \item for every $Y \in \cS_j$, there exists $X \in \cS_i$ such that $Y \subseteq X$,
  \item $X\cap Y=\emptyset$ for every $X \in \cS_i$ and $Y \in \cS_j$.
\end{itemize}
For a subpartition $\cS$, we use $|\cS|$ to denote the number of classes, and $\cS$ is called an \emph{odd subpartition} if $|\cS|$ is odd.
\begin{theorem}[Schrijver \cite{Schr00}]\label{thm:mftdi}
  For a mixed graph $G=(V,E \cup A)$ and a vertex $v \in V$, let $\delta^{\mathrm{head}}(v)$ denote the union of the set of edges in $E$ incident to $v$ and the set of arcs in $A$ with head $v$. The following is a TDI description of the convex hull of matching forests in a mixed graph $G=(V,E \cup A)$.
  \begin{align}
  x_e &\geq 0\ \ \text{for every $e\in E\cup A$} \label{eq:mftdi1}\\
  x (\delta^{\mathrm{head}}(v))& \leq 1\ \ \text{for every $v\in V$}\label{eq:mftdi2}\\
  x(E[\cup \cS ])+\sum_{Z \in \cS} x(A[Z]) &\leq |\cup \cS |-\lceil |\cS|/2\rceil\ \ \text{for every subpartition $\cS$ of $V$}.\label{eq:mftdi3}
  \end{align}
Considering the maximization problem for some cost function $c:E \cup A \to \Z$, there is an integer optimal dual solution such that the support of the dual variables $y$ corresponding to \eqref{eq:mftdi3} is laminar and consists of odd subpartitions.
\end{theorem}

In general, it is known that a TDI description remains TDI when some inequalities are replaced by equalities \cite{SchrBook}.
By this fact, Theorem~\ref{thm:mftdi} implies the following TDI description of perfect matching forests, where \eqref{eq:pmftdi3} is obtained by
subtracting \eqref{eq:mftdi3} from the summation of \eqref{eq:pmftdi2} on $\cup \cS$.

\begin{corollary}\label{cor:pmftdi}
For a mixed graph $G=(V,E \cup A)$, the following is a TDI description of the convex hull of perfect matching forests.
 \begin{align}
  x_e &\geq 0 \ \ \text{for every $e\in E\cup A$}\label{eq:pmftdi1}\\
  x (\delta^{\mathrm{head}}(v))& = 1\ \ \text{for every $v\in V$}\label{eq:pmftdi2}\\
  x(E[\cup \cS ])+x(\delta_E(\cup \cS))+\sum_{Z \in \cS} x(\delta^{in}_A(Z)) &\geq \lceil |\cS|/2\rceil\ \ \text{for every subpartition $\cS$ of $V$}.\label{eq:pmftdi3}
  \end{align}
For any cost function $c:E \cup A \to \Z$, there is an integer optimal dual solution such that the support of the dual variables $y$ corresponding to \eqref{eq:pmftdi3} is laminar and consists of odd subpartitions.
\end{corollary}

\subsection{Algorithmic and Polyhedral Properties of Mixed Edge Covers}\label{sec:mic-property}

%In an inclusionwise minimal mixed edge cover, the components of $F \cap E$ are stars, and $F\cap A$ is a branching whose roots are exactly the endpoints of edges in $F \cap E$.
We first show that there is a close relationship between mixed edge covers and perfect matching forests in a modified graph. This allows us to find a minimum weight mixed edge cover in strongly polynomial time, and to give a TDI description of the convex hull of mixed edge covers.

Given a mixed graph $G=(V, E \cup A)$ with weights $w:E \cup A \to \R_+$,
%\todo{should be nonnegative?}
we construct an auxiliary mixed graph $H=(V \cup V',E\cup A \cup E' \cup A')$ with costs $c$ on $E\cup A\cup E'\cup A'$. Let $V'$ be a copy of $V$, and let $E'$ be the perfect matching between corresponding vertices of $V$ and $V'$, with costs
$c(vv'):=\min_{uv \in E} w(uv)$ (the cost is infinite if there is no such edge). For $uv \in E \cup A$, let $c(uv)=w(uv)$. Finally, let $A'$ consist of arcs $uv'$ for every $u \in V$ and $v' \in V'$, with cost $c(uv')=0$.

\begin{lemma}\label{lem:pfm-mec}
If $G$ has a mixed edge cover, then the minimum weight of a mixed edge cover in $G$ equals the minimum cost of a perfect matching forest in $H$.
\end{lemma}
\begin{proof}
  Let $F$ be a minimum weight mixed edge cover in $G$. We may assume that $E \cap F$ is a disjoint union of stars and $F \cap A$ is a branching whose roots are exactly the endpoints of $E \cap F$. Let $S$  be a star component of $E \cap F$ with center $s$ of degree at least 2. Remove all but one edges of $S$ from $F$, and for every removed edge $sv$, add $vv'$ to $F'$. Do this for every star component of $E \cap F$ with at least 2 edges, and then add arbitrary incoming arcs to the remaining isolated vertices in $V'$. The resulting $F'$ is a perfect matching forest and $c(F')\leq w(F)$.

  Conversely, let $F'$ be a minimum weight perfect matching forest in $H$. For every edge $vv' \in E' \cap F'$, replace $vv'$ by a minimum weight edge in $E$ incident to $v$. Remove all arcs in $A'$. The resulting edge set $F$ is a mixed edge cover in $G$ such that $w(F) \leq c(F')$.
\end{proof}

Combining Lemma~\ref{lem:pfm-mec} with Theorem~\ref{thm:mfalgo} yields the following.
\begin{theorem}
  There is a strongly polynomial-time algorithm to find a minimum weight mixed edge cover.
\end{theorem}

Using the same auxiliary graph $H$ and Corollary~\ref{cor:pmftdi}, we can obtain the following TDI description of mixed edge covers.
The proof is provided in Section~\ref{sec:TDI}.
\begin{theorem}\label{thm:mectdi}
  The following is a TDI description of mixed edge covers:
  \begin{align*}
  1\geq x_e &\geq 0\ \ \text{for every $e\in E\cup A$}\\
  x(E[\cup \cS ])+x(\delta_E(\cup \cS))+\sum_{Z \in \cS} x(\delta^{in}_A(Z)) &\geq \lceil |\cS|/2\rceil\ \ \text{for every subpartition $\cS$ of $V$}.
  \end{align*}
\end{theorem}

\subsection{Proof of TDIness of the Mixed Edge Cover System}\label{sec:TDI}
%\todo{ We should add arguments for general weights.}\\
Let $G=(V, E \cup A)$ be a mixed graph with edge weights $w:E \cup A \to \Z_+$. We assume that $G$ has a mixed edge cover. We construct the auxiliary mixed graph $H=(V \cup V',E\cup A \cup E' \cup A')$ and cost function $c$ as in Section \ref{sec:mic-property}.
Consider the dual of the linear program \eqref{eq:pmftdi1}--\eqref{eq:pmftdi3} for the auxiliary graph $H$ and the cost function $c$:
\begin{align}
  \max \sum \{\lceil |\cS|/2\rceil y_{\cS} : \cS \text{ is a }  &\text{subpartition of  $V \cup V'$}\} - \sum_{v \in V \cup V'} \pi_v& \notag\\
  -\pi_u-\pi_v+\sum \{y_{\cS}: \{u,v\} \cap \cup \cS \neq \emptyset \}  &\leq c(uv)\ \ \text{for every $uv \in E \cup E'$}& \label{eq:edge}\\
  -\pi_v+ \sum \{y_{\cS}: uv \in \delta^{in}(Z) \text{ for some } Z \in \cS\}  &\leq c(uv)\ \ \text{for every $uv \in A \cup A'$}&\label{eq:arc}\\
  \pi_v  &\geq 0\ \ \text{for every $v \in V \cup V'$}&\\
  y_{\cS}  &\geq 0\ \ \text{for every subpartition $\cS$ of $V \cup V'$}.&
  \end{align}

By Corollary \ref{cor:pmftdi}, there is an integral optimal dual solution $(\pi,y)$ such that the support of $y$ is laminar and consists of odd subpartitions.

\begin{lemma}\label{lem:pizero}
The dual linear program for $(H,c)$ has an integral optimal solution $(\pi,y)$ such that the support of $y$ is laminar, it consists of subpartitions disjoint from $V'$, and $\pi \equiv 0$.
\end{lemma}
\begin{proof}
Consider an integral optimal dual solution $(\pi,y)$ where the support of $y$ is laminar and the value $\sum_{u\in V \cup V'} \pi(u)$ is minimal. Let us call a subpartition $\cS$ \emph{positive} if $y_{\cS}>0$. Since the support of $y$ is laminar, each $u\in V \cup V'$ is either uncovered by positive subpartitions, or there is a minimal positive subpartition $\cS$ such that $u \in \cup \cS$. In the latter case, $\cS$ is called the minimal positive subpartition covering $u$. An edge $uv \in E \cup E'$ is called \emph{tight} if \eqref{eq:edge} for $uv$ is satisfied with equality.

\begin{claim}\label{cl:tdi1}
$\pi_{v'}=0$ for every $v' \in V'$.
\end{claim}
\noindent {\em Proof.} Suppose for contradiction that $\pi_{v'}>0$, and consider the following cases.
\begin{itemize}
\item If neither $v$ nor $v'$ is covered by a positive subpartition, then we can decrease $\pi_{v'}$ by 1.
\item Suppose that $v'$ is not covered by a positive subpartition, and the minimal positive subpartition covering $v$ is $\cS$. Let $Z$ be the class of $\cS$ containing $v$, and let $\cS'$ be the subpartition obtained from $\cS$ by removing the class $Z$. We decrease $y_{\cS}$ and $\pi_{v'}$ by 1, and increase $y_{\cS'}$ by 1. This is still a feasible dual solution, because $\eqref{eq:edge}$ still holds for $vv'$, and \eqref{eq:arc} holds for any arc $uv'$ since $v'$ is not covered by a positive subpartition. The objective value does not decrease but $\sum_{u \in V \cup V'} \pi_u$ decreases.

\item Let $\cS$ be the minimal positive subpartition covering $v'$, and let $Z$ be the class of $\cS$ containing $v'$. Suppose that $v \notin \cup \cS$ or $v \in Z$. Let $\cS'$ be the subpartition obtained from $\cS$ by removing the class $Z$. We can decrease $y_{\cS}$ and $\pi_{v'}$ by 1, and increase $y_{\cS'}$ by 1 as in the previous case.

\item Let $\cS$ be the minimal positive subpartition covering $v'$, let $Z$ be the class of $\cS$ containing $v'$, and let $Y$ be the class containing $v$. Let $\cS'$ be the subpartition obtained from $\cS$ by removing the classes $Y$ and $Z$. We get a feasible dual solution by decreasing $y_{\cS}$ and $\pi_{v'}$ by 1, and increasing $y_{\cS'}$ by 1. The objective value remains the same.
\end{itemize}
In all cases, we obtained an optimal dual solution where $\sum_{u \in V \cup V'} \pi_u$ is smaller, contradicting the choice of $(y,\pi)$. \hfill $\diamond$

\begin{claim}\label{cl:tdi2}
$\pi_u=0$ for every $u \in V$.
\end{claim}
\noindent {\em Proof.} First, we consider the case when no positive subpartition covers $u$. Since $\pi_{v'}=0$ for every $v' \in V'$ by the previous Claim,
%\ref{cl:tdi1}
\eqref{eq:arc} for the arcs $uv'$ implies that positive subpartitions are disjoint from $V'$. If there is no tight edge $uv \in E$, then we can just decrease $\pi_u$ by 1. Suppose that there is a tight edge $uv \in E$, i.e.\
$ -\pi_u-\pi_v+\sum \{y_{\cS}: v \in \cup \cS \}  = c(uv)$.
Since $c(vv') \leq c(uv)$, $\eqref{eq:edge}$ for $vv'$ implies that
$-\pi_{v'}-\pi_v+\sum \{y_{\cS}: v \in \cup \cS \}  \leq c(uv)$.
Thus $\pi_u>0$ implies $\pi_{v'}>0$, contradicting the previous Claim.
%\ref{cl:tdi1}

Let now $\cS$ be the minimal positive subpartition covering $u$, and let $Z$ be the class of $\cS$ containing $u$. If $u' \in \cup \cS$, then $u' \in Z$, otherwise \eqref{eq:arc} would be violated for the arc $uu'$. Let $\cS'$ be the subpartition obtained from $\cS$ by removing the class $Z$. If there is no tight edge $uv \in E$ with $v \in \cup\cS\setminus Z$, then we can decrease $y_{\cS}$ and $\pi_{u}$ by 1, and increase $y_{\cS'}$ by 1.

Suppose that there is a tight edge $uv \in E$ with $v \in \cup \cS\setminus Z$. Every positive subpartition covering $u$ also covers $v$, so tightness implies
$-\pi_u-\pi_v+\sum \{y_{\cS}: v \in \cup \cS \}  = c(uv)$.
Since $c(vv') \leq c(uv)$, $\eqref{eq:edge}$ for $vv'$ implies
$-\pi_{v'}-\pi_v+\sum \{y_{\cS}: v \in \cup \cS \}  \leq c(uv)$.
Thus $\pi_u>0$ implies $\pi_{v'}>0$, contradicting the previous Claim. \hfill $\diamond$
%\ref{cl:tdi1}
\medskip

The two Claims
% \ref{cl:tdi1} and \ref{cl:tdi2}
together show that $\pi \equiv 0$, as required. To show that positive subpartitions can be assumed to be disjoint from $V'$, observe that if $v' \in V'$ is covered by a positive subpartition, then the class containing $v'$ must be a superset of $V$, otherwise \eqref{eq:arc} is violated for some arc $uv'$. We can replace this class by $V$ and still get a feasible dual solution.
\end{proof}
\begin{proof}[Proof of Theorem~\ref{thm:mectdi}]
%Since we are considering the upper hull, the primal optimum is unbounded if some weights are negative. Therefore, dual integrality has to be shown only for nonnegative integer weights.
Let $\rho_w(G)$ denote the minimum weight of a mixed edge cover in $G$ for weight function $w$.
First, we prove dual integrality for nonnegative integer weights.
Given a mixed edge cover problem instance $G=(V, E \cup A)$ with edge weights $w:E \cup A \to \Z_+$, we construct the auxiliary mixed graph $H$ and cost function $c$ as above. By Lemma \ref{lem:pfm-mec}, $\rho_w(G)$ equals the minimum cost of a perfect matching forest in $H$.  By Lemma \ref{lem:pizero}, the latter problem has an integer optimal dual solution $(y,\pi)$ where $\pi \equiv 0$ and every positive subpartition is disjoint from $V'$. Since $y$ is a feasible dual solution to the mixed edge cover system for $G$ and its objective value equals $\rho_w(G)$, it is an optimal dual solution.

Consider now the case when $w$ has some negative values. Write $w$ as $w=w^+-w^-$, where $w^+$ is the positive part of $w$ and $w^-$ is the negative part. Clearly, $\rho_{w^+}(G)-\rho_w(G)=w^-(E \cup A)$. Let $y$ be the optimal integer dual solution for $w^+$, obtained as above. For $e \in E \cup A$, let $z_e$ denote the dual variable corresponding to the condition $x_e\leq 1$. If we set $z:=w^-$, then $(y,z)$ is a feasible integer dual solution for $w$ and its objective value equals $\rho_{w^+}(G)+w^-(E \cup A)=\rho_w(G)$, so it is an optimal dual solution.
\end{proof}

\section{Equitable Partitions into Matching Forests}\label{sec:equitable-mf}
In this section, we consider equalization of matching forests.
We provide specific construction methods for the partitions required in Theorems~\ref{thm:mf1} and \ref{thm:mf2}.

Our construction is based on repeated application of operations that equalize a pair of matching forests.
Recall that a matching forest consists of a branching $B$ and a matching $M$ such that $\covered(M)\subseteq R(B)$ (see Proposition~\ref{prop:mf-chara}).
For equalization of edge-size, we want to perform exchanges along alternating paths on edges,
but at the same time we have to modify the arc parts so that the resulting root sets $R'$ and edge sets $M'$ satisfy $\covered(M')\subseteq R(B')$ again.
To cope with this issue, we invoke the following result of Schrijver on root exchange of branchings.

\begin{lemma}[Schrijver \cite{Schr00}]\label{lem:Schrijver}
Let $B_1$ and $B_2$ be branchings and let $R(B_1)$ and $R(B_2)$ denote their root sets.
Let $R'_1$ and $R'_2$ be vertex sets satisfying
$R'_1\cup R'_2=R(B_1)\cup R(B_2)$ and $R'_1\cap R'_2=R(B_1)\cap R(B_2)$.
Then, $B_1\cup B_2$ can be re-partitioned into branchings $B'_1$ and $B'_2$
with $R(B'_1)=R'_1$ and $R(B'_2)=R'_2$ if and only if
each strong component without entering arc (i.e., each source component of $B_1\cup B_2$) intersects both $R'_1$ and $R'_2$.
\end{lemma}

This lemma will also be used for the equalization of mixed edge covers in Section~\ref{sec:equitable-mec}.

\subsection{Operations for a Pair of Matching Forests}

The following two lemmas are the key to the proof of  Theorems~\ref{thm:mf1} and \ref{thm:mf2}.
As in those theorems, for a matching forest $F'_i\subseteq E\cup A$, we use the notations $M'_i:= F'_i\cap E$ and $B'_i:= F'_i\cap A$.
\begin{lemma}\label{lem:mf1}
Let $G=(V, E \cup A)$ be a mixed graph that is the disjoint union of two matching forests $F_1$, $F_2$.
Then $G$ can be partitioned into two matching forests $F'_1$, $F'_2$ such that
\mbox{$||F'_1|-|F'_2|| \leq 1$} and $||F'_1|-|F'_2||+||M'_1|-|M'_2|| \leq 2$.
\end{lemma}
\begin{lemma}\label{lem:mf2}
Let $G=(V, E \cup A)$ be a mixed graph that is the disjoint union of two matching forests $F_1$, $F_2$.
Then $G$ can be partitioned into two matching forests $F'_1$, $F'_2$
such that \mbox{$||M'_1|-|M'_2|| \leq 1$} and $||F'_1|-|F'_2||+||M'_1|-|M'_2|| \leq 2$.
\end{lemma}
In the following, we give a combined proof of the two lemmas.

\begin{proof}
To construct the required matching forests, we introduce four equalizing operations.
\begin{claim}
It is possible to implement the following four operations on disjoint matching forests
$F_1, F_2$, that repartition $F_1\cup F_2$ into matching forests $F'_1, F'_2$ with the properties below.
%Here we use abbreviations $|F_1,F_2|:=|F_1|-|F_2|$ and $|N_1,N_2|:=|N_1|-|N_2|$.
\begin{enumerate}
\setlength{\itemindent}{0mm}
\setlength{\leftskip}{15mm}
\setlength{\itemsep}{2.5mm}
\item[Operation 1.] If $|M_1|-|M_2| > 0$ and $|F_1|-|F_2|\geq 0$, it returns  $F'_1, F'_2$ such that
$|M'_1|-|M'_2|-(|M_1|-|M_2|)=-2$ and $|F'_1|-|F'_2|-(|F_1|-|F_2|)\in\{0,-2\}$.

\item[Operation 2.]  If $|M_1|-|M_2| > 0$ and $|F_1|-|F_2|\leq 0$, it returns  $F'_1, F'_2$ such that
$|M'_1|-|M'_2|-(|M_1|-|M_2|)=-2$ and  $|F'_1|-|F'_2|-(|F_1|-|F_2|)\in\{0,2\}$.

\item[Operation 3.] If $|F_1|-|F_2| > 0$ and $|M_1|-|M_2|\geq 0$, it returns  $F'_1, F'_2$ such that
$|F'_1|-|F'_2|-(|F_1|-|F_2|)=-2$ and $|M'_1|-|M'_2|-(|M_1|-|M_2|)\in\{0,-2\}$.

\item[Operation 4.] If $|F_1|-|F_2| > 0$ and $|M_1|-|M_2|\leq 0$, it returns  $F'_1, F'_2$ such that
$|F'_1|-|F'_2|-(|F_1|-|F_2|)=-2$ and $|M'_1|-|M'_2|-(|M_1|-|M_2|)\in\{0,2\}$.

\end{enumerate}
\end{claim}
We postpone the proof of this claim and complete the proof of the lemmas relying on it.
Note that we also have Operations 1',2',3',4' by switching the roles of $F_1$ and $F_2$.
To prove Lemma~\ref{lem:mf1}, we repeat updating $F_1, F_2$ in the following manner:
\begin{itemize}
\item If $||M_1|-|M_2|| > 2$,
apply Operation 1, 1', 2, or 2' depending on the signs of $|M_1|-|M_2|$ and $|F_1|-|F_2|$, and update $F_1, F_2$ with $F'_1$, $F'_2$.
\smallskip
\item If $||M_1|-|M_2|| \leq 2$ and $||F_1|-|F_2||>1$,
apply Operation 3, 3', 4,  or 4' depending on the signs of $|M_1|-|M_2|$ and $|F_1|-|F_2|$, and update $F_1, F_2$.
\end{itemize}
Note that $||M_1|-|M_2||$ decreases when Operation 1, 1', 2, or 2' is applied.
Also, when Operation 3, 3', 4, or 4' is applied, $||F_1|-|F_2||$ decreases while $||M_1|-|M_2||\leq 2$ is preserved.
Therefore,  we finally obtain $||M_1|-|M_2||\leq 2$ and $||F_1|-|F_2||\leq 1$.
Then $||F_1|-|F_2||+||M_1|-|M_2|| \leq 2$ if $||M_1|-|M_2||<2$ or $||F_1|-|F_2||<1$,
while otherwise we can apply Operation 1, 1', 2, or 2' to update $F_1$ and $F_2$ so that $||M_1|-|M_2||=0$ and $||F_1|-|F_2||=1$.
Thus, we have $||F_1|-|F_2||+||M_1|-|M_2|| \leq 2$, and Lemma~\ref{lem:mf1} is proved.
Lemma~\ref{lem:mf2} can be shown similarly by swapping the roles of $M_i$ and $F_i$
and of Operations 1--2 and 3--4.
\end{proof}

Here we prove the postponed claim.

\paragraph{Proof of the Claim.} Let $R_i:=R(B_i)$ for each $i=1,2$. Note that
\[|F_i|= |M_i|+|B_i|=|M_i|+|V|-|R_i|=|V|-|M_i|-|R_i\setminus\covered(M_i)|.\]

We construct an auxiliary undirected graph $G^*=(V^*, E^*)$. For every node $v \in (R_1\setminus\covered(M_1)) \cup (R_2\setminus\covered(M_2))$, we add a new node $v^{\bullet}$.
The edge set $E^*$ consists of two disjoint matchings $M^*_1$ and $M^*_2$,
where
\begin{align*}
%V^*&=\set{v^{\rm r} |v\in V}\cup \set{v_e|i\in \{1,2\},~ e=uv\in N_i, ~e=\pi_i(u)\neq \pi_i(v)}\\
M^*_i&=M_i\cup M^{\bullet}_i,\\
M^{\bullet}_i&=\set{v^{\bullet}v | v\in R_i\setminus\covered(M_i)}.
\end{align*}
By definition, $M^*_i$ is a matching, and $|F_i|=|V|-|M^*_i|$. If a node $v \in V$ is covered by both $M^*_1$ and $M^*_2$, then $v \notin \covered(B_1 \cup B_2)$, so $v$ is a singleton source component in $B_1 \cup B_2$.

For each source component $S$ of $B_1 \cup B_2$ in the original graph,
if there are $u,v\in S$ such that $u\in R_1\setminus R_2$ and $v\in R_2\setminus R_1$,
take such a pair $(u,v)$ and contract $u$ and $v$ in $G^*$. Let $V^*$ be the resulting node set.
After this operation, $M^*_1$ and $M^*_2$ are still matchings, and hence
$E^*=M^*_1\cup M^*_2$ can be partitioned into alternating cycles and paths. Note that a node $v^{\bullet}$ is either the end-node of a path, or it is in the alternating 2-cycle $v^{\bullet}v$ (the latter occurs when $v \in (R_1\setminus\covered(M_1)) \cap (R_2\setminus\covered(M_2))$). This means that edges in $M^{\bullet}_i$ appear only at the end of paths and in the above-mentioned 2-cycles.
(See an example in Fig.~\ref{fig4}.)
\begin{figure}[htb]
\begin{center}
   \includegraphics[width=95mm]{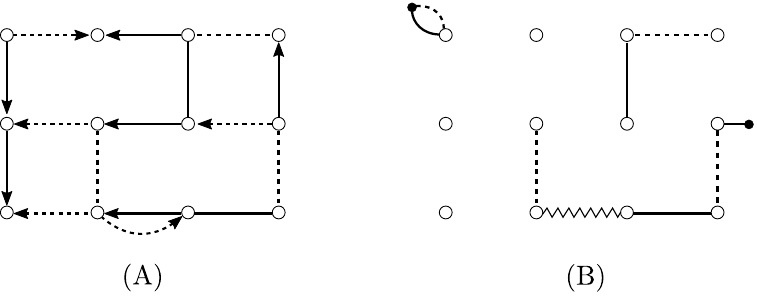}
\caption{%\fontsize{10pt}{0pt}\selectfont
(A) A graph $G=(V,F_1\cup F_2)$. Thick and dashed lines represent $F_1$ and $F_2$, respectively.
(B)~The auxiliary graph $G^*=(V^*, M^*_1\cup M^*_2)$.  Vertices of types $v$ and $v^{\bullet}$ are represented by white and black circles, respectively. The zigzag line means a contraction. The edge set is partitioned into a 2-cycle and two paths.}
\label{fig4}
\end{center}
\end{figure}

Depending on the types of the first and last edges, there are ten types of paths shown in Table~\ref{table:types-mf}.
For each path $P$, let $m(P):=|P\cap M_1|-|P\cap M_2|$
and $f(P):=|P\cap M^*_2|-|P\cap M^*_1|$. We see that these values depend only on the type of $P$.
\begin{table}[hbtp]
  \caption{Types of Alternating Paths}
  \label{table:types-mf}
  \centering
  \begin{tabular}{|c|c|c|c|}
    \hline
    type &  end-edges  & $m(P)$ & $f(P)$ \\
    \hline
	1 & $M^{\bullet}_1, M^{\bullet}_1$ & -1 & -1 \\
	2 & $M_1, M^{\bullet}_1$ & 0 & -1 \\
	3 & $M_1, M_1$ & 1 & -1 \\
	\hline
	4 & $M^{\bullet}_2, M^{\bullet}_2$ & 1 & 1 \\
	5 & $M_2, M^{\bullet}_2$ & 0 & 1 \\
	6 & $M_2, M_2$ & -1 & 1 \\
	\hline
	7 & $M^{\bullet}_1, M^{\bullet}_2$ & 0 & 0 \\
	8 & $M_1, M^{\bullet}_2$ & 1 & 0 \\
	9 & $M^{\bullet}_1, M_2$ & -1 & 0 \\
	10 & $M_1, M_2$ & 0 & 0 \\
    \hline
  \end{tabular}
\end{table}
Now we show the following statement.
\begin{description}
\item[($\star$)\!] If $P_1,\dots,P_k$ is a set of alternating paths in $G^*$,
then we can partition $F_1\cup F_2$ into two matching forests $F'_1$ and $F'_2$ so that
$|M'_1|-|M'_2|=|M_1|-|M_2|-2\sum_{j=1}^k m(P_j)$ and $|F'_1|-|F'_2|=|F_1|-|F_2|-2\sum_{j=1}^k f(P_j)$.
\end{description}
To obtain $F'_1$ and $F'_2$, we first define edge sets $M'_1$, $M'_2$ and root sets $R'_1$, $R'_2$, whose validity we will show.
Let $P=P_1 \cup \dots \cup P_k$, and $P'=P \cap (M_1\cup M_2)$.
For each $i=1,2$, define $M'_i:=M_i\Delta P'$.
Then
\begin{equation}
|M'_1|-|M'_2|=|M_1|-|M_2|-2(|P\cap M_1|-|P\cap M_2|)=|M_1|-|M_2|-2\sum_{j=1}^{k} m(P_j).\label{eq:mf_val_M}
\end{equation}
Let $Q'_i=M^{\bullet}_i \Delta P$, and let
\[R'_i:=\covered(M'_i) \cup \set{v\in V| v^{\bullet}v\in Q'_i} \text{ for each $i=1,2$}.\]
Note that $R_1\cap R_2=R'_1\cap R'_2$ and $R_1\cup R_2=R'_1\cup R'_2$.
Moreover, each strong component $S$ of $B_1\cup B_2$ intersects with $R'_1$ and $R'_2$ as we have contracted $u$ and $v$ in $G^*$
for a pair $u,v\in S$ with $u\in R_1\setminus R_2$ and $v\in R_2\setminus R_1$.
Therefore, by Lemma~\ref{lem:Schrijver}, we can partition $B_1\cup B_2$ into branchings $B'_1$ and $B'_2$ such that $R(B'_1)=R'_1$ and $R(B'_2)=R'_2$.
Define $F'_1:=M'_1\cup B'_1$ and $F'_2:=M'_2\cup B'_2$.
By Proposition~\ref{prop:mf-chara}, $F'_i$ is a matching forest.
Also, as $|F_i|=|M_i|+|V|-|R_i|$, the definition of $f(P)$ implies $|F'_1|-|F'_2|-(|F_1|-|F_2|)=|M'_1|-|M'_2|-(|M_1|-|M_2|)+|R'_2|-|R'_1|-(|R_2|-|R_1|)=-2\sum_{j=1}^{k} f(P_j)$.
Together with \eqref{eq:mf_val_M},  the proof of ($\star$) is completed.
\medskip

By ($\star$), for the implementation of the operations,
it suffices to show the existence of paths with suitable $m(P)$ and $f(P)$ values.
Recall that $M^*_1\cup M^*_2$ is partitioned into alternating paths and cycles;
let $\mathcal{P}$ and $\mathcal{C}$ be those collections of paths and cycles.
Then $|M_1|-|M_2|$ is the sum of the two values
$\sum_{P\in \mathcal{P}} m(P)$ and
$\sum_{C\in\mathcal{C}} m(C)$, where the latter is $0$ as each cycle has even length.
Thus, $|M_1|-|M_2|=\sum_{P\in \mathcal{P}} m(P)$. Similarly, we obtain $|F_1|-|F_2|=|M^*_2|-|M^*_1|=\sum_{P\in \mathcal{P}} f(P)$.
Because each type defines the values of $m(P)$ and $f(P)$ as in Table~\ref{table:types-mf},
we have the following equations,
where we denote by $p(t)$ the number of paths of type $t\in\{1,2,\dots, 10\}$:
\begin{align}
|M_1|-|M_2|&=p(3)+p(4)+p(8)-p(1)-p(6)-p(9)\label{eq:edge-num_mf}\\
|F_1|-|F_2|&=p(4)+p(5)+p(6)-p(1)-p(2)-p(3).\label{eq:total-num_mf}
\end{align}
Now we implement Operations 1--4 in the claim.

\begin{itemize}
\setlength{\leftskip}{-3mm}
\item Operation 1. We have $|M_1|-|M_2|>0$ and $|F_1|-|F_2|\geq 0$.
Since \eqref{eq:edge-num_mf} is positive, at least one of $p(3), p(4), p(8)$ is positive.
If $p(4)>0$ or $p(8)>0$, then there is a path of type 4 or 8. By exchange along such a path, we obtain $F'_1$ and $F'_2$ with the desired properties.
In the remaining case, $p(4)=p(8)=0$. The positivity of \eqref{eq:edge-num_mf} implies $p(3)-p(6)>0$, and hence $p(4)+p(6)-p(3)<0$.
As \eqref{eq:total-num_mf} is nonnegative, we have $p(5)>0$.
Exchange along a pair of paths of types 3 and 5 yields the desired $F'_1$, $F'_2$ by ($\star$).
\medskip

\item Operation 2. We have $|M_1|-|M_2|>0$ and $|F_1|-|F_2|\leq 0$.
Since \eqref{eq:edge-num_mf} is positive, at least one of $p(3), p(4), p(8)$ is positive.
If $p(3)>0$ or $p(8)>0$, then there is a path of type 3 or 8. By exchange along such a path, we obtain $F'_1$ and $F'_2$ with the desired properties.
In the remaining case, $p(3)=p(8)=0$. Then the positivity of \eqref{eq:edge-num_mf} implies $p(4)-p(1)>0$, and hence $p(4)-p(1)-p(3)>0$.
As \eqref{eq:total-num_mf} is nonpositive, we have $p(2)>0$.
Exchange along a pair of paths of types 2 and 4 yields the desired $F'_1$, $F'_2$ by ($\star$).
\medskip

\item Operation 3. We have $|M_1|-|M_2|\geq 0$ and $|F_1|-|F_2|> 0$.
Since \eqref{eq:total-num_mf} is positive, at least one of $p(4), p(5), p(6)$ is positive.
If $p(4)>0$ or $p(5)>0$, then there is a path of type 4 or 5. By exchange along such a path, we obtain $F'_1$ and $F'_2$ with the desired properties.
In the remaining case, $p(4)=p(5)=0$. Then the positivity of \eqref{eq:total-num2} implies $p(6)-p(3)>0$, and hence $p(3)+p(4)-p(6)<0$.
As \eqref{eq:edge-num_mf} is nonnegative, we have $p(8)>0$.
Exchange along a pair of paths of types 6 and 8 yields the desired $F'_1$, $F'_2$ by ($\star$).
\medskip

\item Operation 4. We have $|M_1|-|M_2|\leq 0$ and $|F_1|-|F_2|> 0$.
Since \eqref{eq:total-num_mf} is positive, at least one of $p(4), p(5), p(6)$ is positive.
If $p(5)>0$ or $p(6)>0$, then there is a path of type 5 or 6. By exchange along such a path, we obtain $F'_1$ and $F'_2$ with the desired properties.
In the remaining case, $p(5)=p(6)=0$. Then the positivity of \eqref{eq:total-num_mf} implies $p(4)-p(1)>0$, and hence $p(4)-p(1)-p(6)>0$.
As \eqref{eq:edge-num_mf} is nonpositive, we have $p(9)>0$.
Exchange along a pair of paths of types 4 and 9 yields the desired $F'_1$, $F'_2$ by ($\star$).
\end{itemize}
Thus, Operations 1--4 are implemented.
\qed

\subsection{Proofs of Theorems~\ref{thm:mf1} and \ref{thm:mf2}}

We prove Theorem~\ref{thm:mf1} using Lemma~\ref{lem:mf1} (the proof of Theorem~\ref{thm:mf2} using Lemma~\ref{lem:mf2} is analogous).
We start with an arbitrary partitioning of $G$ into $k$ matching forests $F_1,\dots,F_k$. We describe a 2-phase algorithm to obtain the required partitioning.

In the first phase, in every step we choose $i$ and $j$ with $|F_i|-|F_j|$ maximal, and use Lemma~\ref{lem:mf1} to replace them by matching forests $F_i'$ and $F_j'$ such that
$-1 \leq |F_i'|-|F_j'| \leq 1$. We repeat this until there is a number $q$ such that $|F_i| \in \{q,q+1\}$ for every $i$.
In each step, at least one of the following is true:
\begin{itemize}
\item $\min_{i \in [k]} |F_i|$ increases while $\max_{i \in [k]} |F_i|$ does not increase,
\item $\max_{i \in [k]} |F_i|$ decreases while $\min_{i \in [k]} |F_i|$ does not decrease,
\item the number of indices $i$ such that $|F_i|$ is minimal or maximal decreases.
\end{itemize}
This shows that the number of steps is polynomial.

In the second phase, we distinguish two cases. Suppose first that each $F_i$ has size $q$. In every step, we choose $i$ and $j$ with $|M_i|-|M_j|$ maximal, and use Lemma~\ref{lem:mf1} to replace $F_i$ and $F_j$ by matching forests $F_i'$ and $F_j'$ such that $|F_i'|=|F_j'|=q$ and $-2 \leq |M'_i|-|M'_j| \leq 2$.
We repeat this until $|M_i|-|M_j| \leq 2$ for every $i,j$. Since each $F_i$ still has the same size, the obtained matching forests also satisfy $|B_i|-|B_j|\leq 2$ for every $i,j$.

Now suppose that not every $F_i$ has the same size. In each step we choose $i$ and $j$ such that $|F_i|=q$, $|F_j|=q+1$, and $||M_i|-|M_j||$ is maximal among these. By Lemma~\ref{lem:mf1}, we can replace $F_i$ and $F_j$ by matching forests $F_i'$ and $F_j'$ such that $||F_i'|-|F_j'||=1$ and $-1 \leq |M'_i|-|M'_j| \leq 1$. We repeat this until $|M_i|-|M_j| \leq 1$ whenever $|F_i| \neq |F_j|$. This also implies that $|M_i|-|M_j| \leq 2$ when $|F_i| = |F_j|$. We can conclude that $|B_i|-|B_j|\leq 2$ for every $i,j$.

The number of steps in the second phase can be bounded similarly as in the first phase. One of the following happens in each step:
\begin{itemize}
\item $\min_{i \in [k]} |M_i|$ increases while $\max_{i \in [k]} |M_i|$ does not increase,
\item $\max_{i \in [k]} |M_i|$ decreases while $\min_{i \in [k]} |M_i|$ does not decrease,
\item the number of indices $i$ such that $|M_i|$ is minimal or maximal decreases.
\end{itemize}
Therefore, the number of steps in the second phase is also polynomial.

%%%%%%%%%%%%%%%%%%%%%%%%%%%%%%%%%%%%%%%%%%%%%%%%%%%%%%%%%%%%%%%%%%%%%%%%%%%%%%%%%%%%%%%%
\section{Equitable Partitions into Mixed Edge Covers}\label{sec:equitable-mec}
In this section, we show how to obtain the mixed edge covers
required in Theorems~\ref{thm:mec1} and \ref{thm:mec2}.
Similarly to the case of matching forests, we repeat equalization of a pair of mixed edge covers.

Recall that a mixed edge cover is characterized by containing a branching $B$ and an edge set $N$ with $R(B)\subseteq \covered(N)$ (see Proposition~\ref{prop:mec-chara}).
To keep the edge parts and the arc parts compatible throughout the construction,
we again utilize Lemma~\ref{lem:Schrijver} of Schrijver.

\subsection{Operations for a Pair of Mixed Edge Covers}
To obtain Theorems~\ref{thm:mec1} and \ref{thm:mec2}, we use the following two lemmas.
As before, for a mixed edge cover $F'_i\subseteq E\cup A$, we write $N'_i:= F'_i\cap E$ and $B'_i:= F'_i\cap A$.
\begin{lemma}\label{lem:mec1}
Let $G=(V, E \cup A)$ be a mixed graph that can be partitioned into two mixed edge covers $F_1$, $F_2$.
Then $G$ contains two disjoint mixed edge covers $F'_1$, $F'_2$ such that
$||F'_1|-|F'_2|| \leq 1$ and $||F'_1|-|F'_2||+||N'_1|-|N'_2|| \leq 2$.
\end{lemma}
\begin{lemma}\label{lem:mec2}
Let $G=(V, E \cup A)$ be a mixed graph that can be partitioned into two mixed edge covers $F_1$, $F_2$.
Then $G$ contains two disjoint mixed edge covers $F'_1$, $F'_2$
such that $||N'_1|-|N'_2|| \leq 1$ and $||F'_1|-|F'_2||+||N'_1|-|N'_2|| \leq 2$.
\end{lemma}
In the following, we give a combined proof of the two lemmas.
\begin{proof}
To construct the required mixed edge covers, we introduce four equalizing operations.

\begin{claim}
It is possible to implement the following four operations, each of which is applied to minimal mixed edge covers
$F_1, F_2$ and repartition $F_1\cup F_2$ into (not necessarily minimal) mixed edge covers $F'_1, F'_2$ with the properties below.
\begin{enumerate}
\setlength{\itemindent}{0mm}
\setlength{\leftskip}{15mm}
\setlength{\itemsep}{2.5mm}
\item[Operation 1.] If $|N_1|-|N_2|>0$ and $|F_1|-|F_2|\geq 0$, it returns  $F'_1, F'_2$ such that
$|N'_1|-|N'_2|-(|N_1|-|N_2|)=-2$ and $|F'_1|-|F'_2|-(|F_1|-|F_2|)\in\{0,-2\}$.

\item[Operation 2.] If $|N_1|-|N_2|>0$ and $|F_1|-|F_2|\leq 0$, it returns  $F'_1, F'_2$ such that
$|N'_1|-|N'_2|-(|N_1|-|N_2|)=-2$ and $|F'_1|-|F'_2|-(|F_1|-|F_2|)\in\{0,2\}$.

\item[Operation 3.] If $|F_1|-|F_2|>0$ and $|N_1|-|N_2|\geq 0$, it returns  $F'_1, F'_2$ such that
$|F'_1|-|F'_2|-(|F_1|-|F_2|)=-2$ and $|N'_1|-|N'_2|-(|N_1|-|N_2|)\in\{0,-2\}$.

\item[Operation 4.] If $|F_1|-|F_2|>0$ and $|N_1|-|N_2|\leq 0$, it returns  $F'_1, F'_2$ such that
$|F'_1|-|F'_2|-(|F_1|-|F_2|)=-2$ and $|N'_1|-|N'_2|-(|N_1|-|N_2|)\in\{0,2\}$.

\end{enumerate}
\end{claim}
We postpone the proof of this claim and give a proof of the lemmas relying on it.
Note that we also have Operations 1',2',3',4' by switching the roles of $F_1$ and $F_2$.
By the assumption, we have two disjoint mixed edge covers $F_1$ and $F_2$.
For Lemma~\ref{lem:mec1}, we repeat updating $F_1, F_2$ in the following manner:
\begin{itemize}
\item If $F_i$ is not minimal, replace it with a minimal mixed edge cover $F'_i\subseteq F_i$.
\smallskip
\item If $F_1$ and $F_2$ are minimal and $||N_1|-|N_2||>2$,
apply Operation 1, 1', \mbox{2, or 2'} depending on the signs of $|N_1|-|N_2|$ and $|F_1|-|F_2|$ and update $F_1, F_2$ with $F'_1$, $F'_2$.
\smallskip
\item If $F_1$ and  $F_2$ are minimal and $||N_1|-|N_2||\leq 2$ and $||F_1|-|F_2||>1$,
apply Operation 3, 3', 4 or 4' depending on the signs of $|N_1|-|N_2|$ and $|F_1|-|F_2|$ and update $F_1, F_2$ with $F'_1$, $F'_2$.
\end{itemize}
Throughout the repetition of updates, $|F_1\cup F_2|$ is monotone decreasing.
Note that $||N_1|-|N_2||$ decreases when Operation 1, 1', 2 or 2' is applied.
Also, when Operation 3, 3', 4 or 4' is applied, $||F_1|-|F_2||$ decreases while  $||N_1|-|N_2||\leq 2$ is preserved.
Thus, $(|F_1\cup F_2|, \max\{||N_1|-|N_2||, 2\}, ||F_1|-|F_2||)$ is lexicographically monotone decreasing and
we finally obtain $||N_1|-|N_2||\leq 2$ and $||F_1|-|F_2||\leq 1$.
Then $||F_1|-|F_2||+||N_1|-|N_2|| \leq 2$ if $||N_1|-|N_2||<2$ or $||F_1|-|F_2||<1$,
while otherwise we can apply Operation 1, 1',2 or 2' to update $F_1$ and $F_2$ so that $||N_1|-|N_2||=0$ and $||F_1|-|F_2||=1$.
Thus, we have $||F_1|-|F_2||+||N_1|-|N_2|| \leq 2$, and Lemma~\ref{lem:mec1} is proved.
Lemma~\ref{lem:mec2} can be shown similarly by swapping the roles of $N_i$ and $F_i$
and that of Operations 1--2 and 3--4.
\end{proof}

Here we prove the postponed claim.

\paragraph{Proof of the Claim.}
By the assumption, we have two disjoint minimal mixed edge covers $F_1$ and $F_2$.
Proposition~\ref{prop:mec-chara} and the minimality of each $F_{i}$ imply that
\begin{itemize}
\item $B_i:=F_{i}\cap A$ forms a branching whose root set $R_i$ satisfies $R_i= \covered(N_{i})$.
\item For every $e\in N_i=F_{i}\cap E$, at least one endpoint is covered only by $e$ in $F_i$. (Hence, $N_{i}$ forms a union of stars.)
\end{itemize}
We construct an auxiliary undirected graph $G^*=(V^*,N^*)$ to find good alternating paths in $N_1\cup N_2$.
First, for each $i=1,2$ and $v\in R_i$, choose any edge $e\in N_i$ incident to $v$, and call it $\pi_i(v)$.
We say that $v$ {\bf chooses} $e$ in $N_i$ if $\pi_i(v)=e$.
As $N_i$ is the union of stars, each $e\in N_{i}$ is chosen by at least one endpoint.
For convenience, we set $\pi_i(v)=\emptyset$ for each $v\in V\setminus R_i$.
The vertex set $V^*$ is given by
\[V^*:=\set{v^{\rm r} |v\in V}\cup \set{v^i_e| i\in\{1,2\},~e=uv\in N_i,~\pi_i(u)=e\neq\pi_i(v)}.\]
Thus, the center of each star is split into multiple vertices (see Fig.~\ref{fig5}).
The edge set $N^*$ consists of two disjoint parts $N^*_1$ and $N^*_2$,
and each $N^*_i$ is defined as
\begin{align*}
%V^*&=\set{v^{\rm r} |v\in V}\cup \set{v_e|i\in \{1,2\},~ e=uv\in N_i, ~e=\pi_i(u)\neq \pi_i(v)}\\
N^*_i&=N^{\circ}_i\cup N^{\bullet}_i,\\
N^{\circ}_i&=\set{u^{\rm r}v^{\rm r}|e=uv\in N_i,~\pi_i(u)=e=\pi_i(v)},\\
N^{\bullet}_i&=\set{u^{\rm r}v^i_e|e=uv\in N_i,~\pi_i(u)=e\neq \pi_i(v)},
\end{align*}
where each edge is an unordered pair, i.e., $uv=vu$.
There is a one-to-one correspondence between $N_{i}$ and $N^*_i$, and hence $|N_i|=|N^*_i|=|N^{\circ}_i|+|N^{\bullet}_i|$.
Also, because each root chooses exactly one edge, we have $|R_i|=2|N^{\circ}_i|+|N^{\bullet}_i|$, and hence $|F_i|=|N_i|+(|V|-|R_i|)=|V|-|N^{\circ}_i|$.
Therefore,
\begin{align}
|N_1|-|N_2|&=|N^*_1|-|N^*_2|=|N^{\circ}_1|+|N^{\bullet}_1|-|N^{\circ}_2|-|N^{\bullet}_2|\label{eq:edge-num}\\
|F_1|-|F_2|&=|N^{\circ}_2|-|N^{\circ}_1|.\label{eq:total-num}
\end{align}
This definition of $G^*$ gives the following property.
\begin{description}
\item[(a)\!] Each vertex of type $v^{\rm r}$ is incident to one edge in $N^*_i$ if $v\in R_i$ and otherwise no edge in $N^*_i$.
Each vertex of type $v_e^i$ is incident to one edge in $N^{\bullet}_i$ and no edge in $N^*\setminus N^{\bullet}_i$.
Therefore, $N^*_i$ is a matching in $G^*$ for each $i=1,2$.
\end{description}
For each source component $S$ of $B_{1}\cup B_{2}$ in the original graph,
if there are $u,v\in S$ such that $u\in R_1\setminus R_2$ and $v\in R_2\setminus R_1$,
take such a pair $(u,v)$ and contract $u^{\rm r}$ and $v^{\rm r}$ in $G^*$.
After this operation, $N^*_1$ and $N^*_2$ are still matchings, and hence
$N^*_1\cup N^*_2$ can be partitioned into alternating cycles and paths.
Note that each path in $G^*$ corresponds to a walk in $G$, which is not necessarily acyclic.
(See an example in Fig.~\ref{fig5}.)
\begin{figure}[htb]
\begin{center}
   \includegraphics[width=95mm]{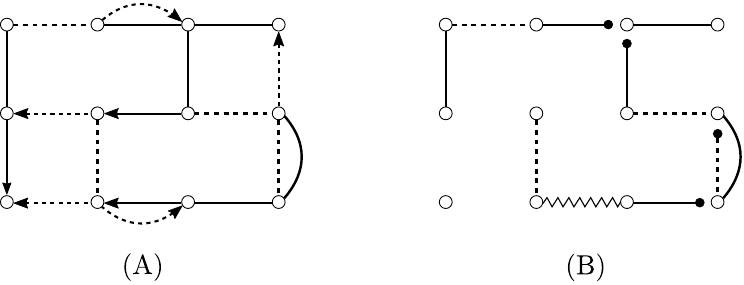}
\caption{%\fontsize{10pt}{0pt}\selectfont
(A) A graph $G=(V,F_1\cup F_2)$. Thick and dashed lines represent $F_1$ and $F_2$, respectively.
(B)~The auxiliary graph $G^*=(V^*, N^*_1\cup N^*_2)$ for some $\pi_1$ and $\pi_2$.  Vertices of types $v^{\rm r}$ and $v^i_e$ are represented by white and black circles, respectively. The zigzag line means a contraction. The edge set is partitioned into four paths.}
\label{fig5}
\end{center}
\end{figure}

By (a), for any path or cycle, all internal vertices are of type $v^{\rm r}$, and hence
all internal edges belong to $N^{\circ}_1\cup N^{\circ}_2$.
Only the first and last edge can belong to $N^{\bullet}_1\cup N^{\bullet}_2$.
Depending on the types of the first and last edges, there are ten types of paths, shown in Table~\ref{table:types-mec}.
For each path $P$, let $n(P):=|P\cap N^*_1|-|P\cap N^*_2|$ %|P\cap N^{\circ}_1|+|P\cap N^{\bullet}_1|-|P\cap N^{\circ}_2|-|P\cap N^{\bullet}_2|$
and $f(P):=|P\cap N^{\circ}_2|-|P\cap N^{\circ}_1|$. We see that these values depend only on the type of $P$.
\begin{table}[hbtp]
  \caption{Types of Alternating Paths}
  \label{table:types-mec}
  \centering
  \begin{tabular}{|c|c|c|c|}
    \hline
    type &  end-edges  & $n(P)$ & $f(P)$ \\
    \hline
	1 & $N^{\bullet}_1, N^{\bullet}_1$ & 1 & 1 \\
	2 & $N^{\circ}_1, N^{\bullet}_1$ & 1 & 0 \\
	3 & $N^{\circ}_1, N^{\circ}_1$ & 1 & -1 \\
	\hline
	4 & $N^{\bullet}_2, N^{\bullet}_2$ & -1 & -1 \\
	5 & $N^{\circ}_2, N^{\bullet}_2$ & -1 & 0 \\
	6 & $N^{\circ}_2, N^{\circ}_2$ & -1 & 1 \\
	\hline
	7 & $N^{\bullet}_1, N^{\bullet}_2$ & 0 & 0 \\
	8 & $N^{\circ}_1, N^{\bullet}_2$ & 0 & 1 \\
	9 & $N^{\bullet}_1, N^{\circ}_2$ & 0 & -1 \\
	10 & $N^{\circ}_1, N^{\circ}_2$ & 0 & 0 \\
    \hline
  \end{tabular}
\end{table}

Now we show the following statement.
\begin{description}
\item[(b)\!]
If $P_1,\dots,P_k$ is a set of alternating paths in $G^*$,
then we can partition $F_1\cup F_2$ into two mixed edge covers $F'_1$ and $F'_2$ so that
$|N'_1|-|N'_2|=|N_1|-|N_2|-2\sum_{j=1}^k n(P_j)$ and $|F'_1|-|F'_2|=|F_1|-|F_2|-2\sum_{j=1}^k f(P_j)$.
\end{description}
To obtain $F'_1$ and $F'_2$, we first define edge sets $N'_1$, $N'_2$ and root sets $R'_1$, $R'_2$, whose validity we will show.
Let $P=P_1 \cup \dots \cup P_k$ and let $P'\subseteq N_1\cup N_2$ be the union of walks in $G$ corresponding to $P$. For each $i=1,2$ define $N'_i:=N_i\Delta P'$.
Then $N'_i$ corresponds to $N^*_i\Delta P$ in $G^*$ and
\begin{equation}
|N'_1|-|N'_2|=|N_1|-|N_2|-2(|P\cap N^*_1|-|P\cap N^*_2|)=|N_1|-|N_2|-2\sum_{j=1}^{k} n(P_j).\label{eq:mec_val_N}
\end{equation}
Let $R'_i:=\set{v\in V|v^{\rm r}\in \covered_{G^*}(N^*_i\Delta P)}$ for each $i=1,2$.
Because both endpoints of each $e\in N^{\circ}_1\cup N^{\circ}_2$ and one endpoint of each $e\in N^{\bullet}_1\cup N^{\bullet}_2$ are of type $v^{\rm r}$,
\begin{align}
|R'_1|-|R'_2|&=|R_1|-|R_2|-2(|P\cap N^*_1|-|P\cap N^*_2|)-2(|P\cap N^{\circ}_1|-|P\cap N^{\circ}_2|)\nonumber\\
&=|R_1|-|R_2|-2\sum_{j=1}^k n(P_j)+2\sum_{j=1}^k f(P_j).\label{eq:mec_val_R}
\end{align}
Note that $R_1\cap R_2=R'_1\cap R'_2$ and $R_1\cup R_2=R'_1\cup R'_2$.
Moreover, each strong component $S$ of $B_1\cup B_2$ intersects with $R'_1$ and $R'_2$ as we have contracted $u^{\rm r}$ and $v^{\rm r}$ in $G^*$
for a pair $u,v\in S$ with $u\in R_1\setminus R_2$ and $v\in R_2\setminus R_1$.
Therefore, by Lemma~\ref{lem:Schrijver}, we can partition $B_1\cup B_2$ into branchings $B'_1$ and $B'_2$ such that $R(B'_1)=R'_1$ and $R(B'_2)=R'_2$.
Define $F'_1:=N'_1\cup B'_1$ and $F'_2:=N'_2\cup B'_2$.
By the definition, each $R'_i$ satisfies $R'_i\subseteq \covered(N'_i)$. Then $F'_i$ is a mixed edge cover by Proposition~\ref{prop:mec-chara}.
Also, by $|F_i|=|N_i|+|B_i|=|N_i|+(|V|-|R_i|)$ and \eqref{eq:mec_val_N}, \eqref{eq:mec_val_R}, we have  $|F'_1|-|F'_2|-(|F_1|-|F_2|)=|N'_1|-|N'_2|-(|N_1|-|N_2|)+|R'_2|-|R'_1|-(|R_2|-|R_1|)=-2\sum_{j=1}^k f(P_j)$.
Together with \eqref{eq:mec_val_N},  this completes the proof of (b).
\medskip

By (b), for the implementation of the operations,
it suffices to show the existence of paths with suitable $n(P)$ and $f(P)$ values.
Recall that  $N^*_1\cup N^*_2$ is partitioned into alternating paths and cycles;
let $\mathcal{P}$ and $\mathcal{C}$ be those collections of paths and cycles.
Then $|N^*_1|-|N^*_2|$ is the sum of two values
$\sum_{P\in \mathcal{P}} n(P)$ and
$\sum_{C\in\mathcal{C}}n(C)$, where the latter is $0$ as each cycle has even length.
Then \eqref{eq:edge-num} implies $|N_1|-|N_2|=|N^*_1|-|N^*_2|=\sum_{P\in \mathcal{P}} n(P)$. A similar argument and \eqref{eq:total-num} imply $|F_1|-|F_2|=|N^{\circ}_1|-|N^{\circ}_2|=\sum_{P\in \mathcal{P}} f(P)$.
Because each type defines the values of $n(P)$ and $f(P)$ as in Table~\ref{table:types-mec},
we have the following equations, where we denote by $p(t)$ the number of paths of type $t\in\{1,2,\dots, 10\}$:
\begin{align}
|N_1|-|N_2|&=p(1)+p(2)+p(3)-p(4)-p(5)-p(6)\label{eq:edge-num2}\\
|F_1|-|F_2|&=p(1)+p(6)+p(8)-p(3)-p(4)-p(9).\label{eq:total-num2}
\end{align}
Now we implement Operations 1--4 in the claim.

\begin{itemize}
\item Operation 1. We have $|N_1|-|N_2|>0$ and $|F_1|-|F_2|\geq 0$.
Because \eqref{eq:edge-num2} is positive, at least one of $p(1), p(2), p(3)$ is positive.
If $p(1)>0$ or $p(2)>0$, then there is a path of type 1 or 2. By exchange along such a path, we obtain $F'_1$ and $F'_2$ with the desired condition.
In the remaining case, $p(1)=0$ and $p(2)=0$. Then the positivity of \eqref{eq:edge-num2} implies $p(3)-p(6)>0$, and hence $p(1)+p(6)-p(3)<0$.
As \eqref{eq:total-num2} is nonnegative, we have $p(8)>0$.
Exchange along a pair of paths of types 3 and 8 yields the desired $F'_1$, $F'_2$  by (b).
\medskip

\item Operation 2. We have $|N_1|-|N_2|>0$ and $|F_1|-|F_2|\leq 0$.
Because \eqref{eq:edge-num2} is positive, at least one of $p(1), p(2), p(3)$ is positive.
If $p(2)>0$ or $p(3)>0$, then there is a path of type 2 or 3. By exchange along such a path, we obtain $F'_1$ and $F'_2$ with the desired condition.
In the remaining case, $p(2)=0$ and $p(3)=0$. Then the positivity of \eqref{eq:edge-num2} implies $p(1)-p(4)>0$, and hence $p(1)-p(3)-p(4)>0$.
As \eqref{eq:total-num2} is nonpositive, we have $p(9)>0$.
Exchange along a pair of paths of types 1 and 9 yields the desired $F'_1$, $F'_2$  by (b).
\medskip

\item Operation 3. We have $|N_1|-|N_2|\geq 0$ and $|F_1|-|F_2|> 0$.
Because \eqref{eq:total-num2} is positive, at least one of $p(1), p(6), p(8)$ is positive.
If $p(1)>0$ or $p(8)>0$, then there is a path of type 1 or 8. By exchange along such a path, we obtain $F'_1$ and $F'_2$ with the desired condition.
In the remaining case, $p(1)=0$ and $p(8)=0$. Then the positivity of \eqref{eq:total-num2} implies $p(6)-p(3)>0$, and hence $p(1)+p(3)-p(6)<0$.
As \eqref{eq:edge-num2} is nonnegative, we have $p(2)>0$.
Exchange along a pair of paths of types 2 and 6 yields the desired $F'_1$, $F'_2$  by (b).
\medskip

\item Operation 4. We have $|N_1|-|N_2|\leq 0$ and $|F_1|-|F_2|> 0$.
Because \eqref{eq:total-num2} is positive, at least one of $p(1), p(6), p(8)$ is positive.
If $p(6)>0$ or $p(8)>0$, then there is a path of type 6 or 8. By exchange along such a path, we obtain $F'_1$ and $F'_2$ with the desired condition.
In the remaining case, $p(6)=0$ and $p(8)=0$. Then the positivity of \eqref{eq:total-num2} implies $p(1)-p(4)>0$, and hence $p(1)-p(4)-p(6)>0$.
As \eqref{eq:edge-num2} is nonpositive, we have $p(5)>0$.
Exchange along a pair of paths of types 1 and 5 yields the desired $F'_1$, $F'_2$  by (b).
\end{itemize}
Thus, Operations 1--4 are implemented.
\qed

\subsection{Proofs of Theorems~\ref{thm:mec1} and \ref{thm:mec2}}
Now we prove Theorems~\ref{thm:mec1} and \ref{thm:mec2} using Lemmas~\ref{lem:mec1} and \ref{lem:mec2}, respectively.
\paragraph{Proof of Theorem~\ref{thm:mec1}.}
By the assumption of the theorem, we have $k$ disjoint mixed edge covers $F_1, F_2,\dots, F_k$ in $G$.
We repeat updating them by the following 2-phase algorithm.

In the first phase, in every step we choose $i$ and $j$ with $|F_i|-|F_j|$ maximal, and use Lemma~\ref{lem:mec1} to replace them
by mixed edge covers $F_i'$ and $F_j'$ such that $||F_i'|-|F_j'|| \leq 1$.
We repeat this until there is a number $q$ such that $|F_i| \in \{q,q+1\}$ for every $i$.
This is achieved in a polynomial number of steps by a similar argument as in the proof of Theorem ~\ref{thm:mf1}, with the additional
observation that there are at most $|E\cup A|$ steps that decrease $|F_1\cup F_2\cup \cdots \cup F_k|$.
%because every step decreases $(|F_1\cup F_2\cup \cdots \cup F_k|, ~\sum_{i,j}(||F_i|-|F_j||))$ lexicographically.

In the second phase, we distinguish two cases.
Suppose first that each $F_i$ has the same size $q$. In every step, we choose $i$ and $j$ with $|N_i|-|N_j|$ maximal,
and use Lemma~\ref{lem:mec1} to replace $F_i$ and $F_j$ with mixed edge covers $F_i'$ and $F_j'$ with
$||F_i'|-|F_j'||\leq 1$ and $||F_i'|-|F_j'||+||N'_i|-|N'_j||\leq 2$.
If $F'_i\cup F'_j$ is a proper subset of $F_i\cup F_j$, we go back to the beginning of the first phase.
Otherwise we continue the second phase, where $|F'_i|=|F'_j|=q$ follows from  $|F'_i|+|F'_j|=|F_i|+|F_j|=2q$.
We repeat this until $||N_i|-|N_j|| \leq 2$ for every $i,j$.
Note that during this phase, the size of every $F_i$ remains $q$.
Therefore, when this phase terminates, we have $||N_i|-|N_j|| \leq 2$ and $||F_i|-|F_j||=0$ for every $i,j$.

Now suppose that not every $F_i$ has the same size at the end of the first phase.
Then, in each step of the second phase we choose $i$ and $j$ such that $|F_i|=q$, $|F_j|=q+1$, and $||N_i|-|N_j||$ is maximal among these.
By Lemma~\ref{lem:mec1}, we can replace $F_i$ and $F_j$ by mixed edge covers $F_i'$ and $F_j'$ such that $||F_i'|-|F_j'||\leq 1$ and $||F_i'|-|F_j'||+||N'_i|-|N'_j||\leq 2$.
If $F'_i\cup F'_j$ is a proper subset of $F_i\cup F_j$, we go back to the beginning of the first phase.
Otherwise we continue the second phase, where $\{|F_i'|,|F_j'|\}=\{q,q+1\}$ follows from $|F_i'|+|F_j'|=|F_i|+|F_j|$, and hence $||N'_i|-|N'_j||\leq 1$. We repeat this until $||N_i|-|N_j|| \leq 1$ whenever $|F_i| \neq |F_j|$.
This also implies that $||N_i|-|N_j|| \leq 2$ when $|F_i| = |F_j|$.

Note that the algorithm goes back to the first phase at most $|E\cup A|$ times
because it decreases $|F_1\cup F_2\cup \cdots \cup F_k|$. Thus the algorithm terminates in a polynomial number of steps,
and we finally obtain $(F_1, F_2,\dots,F_k)$ such that, for every $i,j\in [k]$, the value of $(|N_i|-|N_j|, |F_i|-|F_j|)$ belongs to
\begin{equation}
\{(0,0), \pm(0,1), \pm(1,0), \pm(1,1), \pm(1,-1), \pm(2,0)\}.\label{eq:NF1}
\end{equation}

We now define a superset $F''_i$ of each $F_i$ so that $(F''_1, F''_2,\dots,F''_k)$ forms a partition of $E\cup A$.
Note that any superset of a mixed edge cover is also a mixed edge cover.
So we care only about  the numbers of edges and arcs in $F''_i\setminus F_i$.

Let $E':=E\setminus(F_1\cup F_2\cup \cdots \cup F_k)$ and $n_E$ be the remainder of the division of $|E'|$ by $k$.
Divide $E'$ into $k$ parts $E'_1, E'_2, \dots, E'_k$ such that
\begin{itemize}
\item $|E'_i|=\lfloor|E'|/k\rfloor+1$ for the smallest $n_E$ members $F_i$ with respect to  $(|F_i|, |N_i|)$,
\item $|E'_i|=\lfloor|E'|/k\rfloor$ for other $F_i$,
\end{itemize}
where the order for $(|F_i|, |N_i|)$ is defined lexicographically.
Let $F'_i:=F_i\cup E'_i$ for each $i\in [k]$.
By the definition of $E'_i$, the condition $|N'_i|-|N'_j|>|N_i|-|N_j|$ implies either (i) $|F_i|-|F_j|<0$ or (ii) $|N_i|-|N_j|\leq 0$ and $|F_i|-|F_j|=0$.
Also, it implies $|F'_i|-|F'_j|>|F_i|-|F_j|$. Then, we can check that, for every $i,j\in[k]$,
the pair $(|N'_i|-|N'_j|, |F'_i|-|F'_j|)$ stays in the set of \eqref{eq:NF1}.

Define $A':=A\setminus(F_1\cup F_2\cup \cdots\cup F_k)$ and let $n_A$ be the remainder of the division of $|A'|$ by $k$.
Divide $A'$ into $A'_1,A'_2,\dots,A'_k$ such that
\begin{itemize}
\item $|A'_i|=\lfloor|A'|/k\rfloor+1$ for the smallest $n_A$ members $F'_i$ with respect to $(|F'_i|, -|N'_i|)$
\item $|A'_i|=\lfloor|A'|/k\rfloor$ for other $F_i$.
\end{itemize}
Let $F''_i:=F'_i\cup A'_i$ for each $i\in [k]$.
Then $(F''_1, F''_2,\dots,F''_k)$ is a partition of $E\cup A$ consisting of $k$ mixed edge covers.
By the definition of $A'_i$, $|F''_i|-|F''_j|>|F'_i|-|F'_j|$ implies either (i) $|F'_i|-|F'_j|<0$ or (ii) $|N'_i|-|N'_j|\geq 0$ and $|F'_i|-|F'_j|=0$.
Also $|N''_i|-|N''_j|=|N'_i|-|N'_j|$ for every $i,j\in[k]$.
Because $(|N'_i|-|N'_j|, |F'_i|-|F'_j|)$ belongs to \eqref{eq:NF1}, we see that $(|N''_i|-|N''_j|, |F''_i|-|F''_j|)$ belongs to
\begin{equation*}
\{(0,0), \pm(0,1), \pm(1,0), \pm(1,1), \pm(1,-1), \pm(2,0), \pm(2,1)\}.
\end{equation*}
%\begin{eqnarray*}
%(|N''_i|-|N''_j|, |F''_i|-|F''_j|)\in \left\{
%\begin{array}{l}
%(0,0), \pm(0,1), \pm(1,0), \pm(1,1), \pm(1,-1), \pm(2,0),\\
%\hspace{53mm}\pm(2,1)
%\end{array}
%\right\}
%\end{eqnarray*}
Note that $|B''_i|-|B''_j|=(|F''_i|-|F''_j|)-(|N''_i|-|N''_j|)$.
Then, for every $i,j\in[k]$ we have
$||F''_i|-|F''_j|| \leq 1$, $||N''_i|-|N''_j|| \leq 2$, and  $||B''_i|-|B''_j|| \leq 2$.
\qed

\paragraph{Proof of Theorem~\ref{thm:mec2}.}
Analogously to the proof of Theorem~\ref{thm:mec1},  using Lemma~\ref{lem:mec2} repeatedly we obtain $k$ disjoint mixed edge covers
$(F_1, F_2,\dots,F_k)$ such that, for every $i,j\in [k]$, the value $(|N_i|-|N_j|, |F_i|-|F_j|)$ belongs to
\begin{equation}
\{(0,0), \pm(0,1), \pm(0,2), \pm(1,0), \pm(1,1), \pm(1,-1)\}.\label{eq:NF2}
\end{equation}
Define $E'_i$ and $F'_i$ as in the proof of Theorem~\ref{thm:mec1} except that we use $(|N_i|, |F_i|)$ instead of $(|F_i|,|N_i|)$.
Then, the condition $|N'_i|-|N'_j|>|N_i|-|N_j|$ implies either (i) $|N_i|-|N_j|<0$ or (ii) $|N_i|-|N_j|=0$ and $|F_i|-|F_j|\leq 0$.
This implies that $(|N'_i|-|N'_j|, |F'_i|-|F'_j|)$ belongs to \eqref{eq:NF2} again.
Define $A'_i$ and $F''_i$ as in the proof of Theorem~\ref{thm:mec1} (in fact, it is sufficient to use order on $|F'_i|$ instead of $(|F'_i|, -|N'_i|)$).
Then $(|N''_i|-|N''_j|, |F''_i|-|F''_j|)$ still belongs to \eqref{eq:NF2}.
Therefore, $F''_1, F''_2,\dots,F''_k$ are mixed edge covers partitioning $E\cup A$ and satisfying
$||F''_i|-|F''_j|| \leq 2$, $||N''_i|-|N''_j|| \leq 1$, and  $||B''_i|-|B''_j|| \leq 2$
for every $i,j\in [k]$.
\qed
\subsection{Remarks on Mixed Covering Forests and on Bibranchings}\label{sec:mcf}

As mentioned in the Introduction, mixed covering forests are hard to equalize as they require acyclicity.
On the other hand, by Proposition~\ref{prop:equivalence}, any mixed edge cover contains some mixed covering forest as a subgraph.
This fact implies packing versions of Theorems~\ref{thm:mec1} and \ref{thm:mec2} for mixed covering forests.

\begin{corollary}\label{cor:mcf1}
Let $G=(V, E \cup A)$ be a mixed graph that contains $k$ disjoint mixed covering forests. Then $G$ contains $k$ disjoint mixed covering forests $F_1,\dots,F_k$
such that, for every $i,j \in [k]$, we have $||F_i|-|F_j|| \leq 1$, $||N_i|-|N_j|| \leq 2$, and  $||B_i|-|B_j|| \leq 2$.
\end{corollary}

\begin{corollary}\label{cor:mcf2}
Let $G=(V, E \cup A)$ be a mixed graph that contains $k$ disjoint mixed covering forests. Then $G$ contains $k$ disjoint mixed covering forests $F_1,\dots,F_k$
such that, for every $i,j \in [k]$, we have $||F_i|-|F_j|| \leq 2$, $||N_i|-|N_j|| \leq 1$, and  $||B_i|-|B_j|| \leq 2$.
\end{corollary}

\begin{proof}
These corollaries are shown by modifying the 2-phase algorithm used in the proofs of Theorems~\ref{thm:mec1} and \ref{thm:mec2}.
When we repeat updates of $(F_1, F_2,\dots, F_k)$,
we also consider the following operation: ``If some $F_i$ is not a mixed covering forest, replace $F_i$ with a mixed covering forest contained in it.''
This additional operation does not violate monotonicity. We obtain the required mixed covering forests when the algorithm terminates.
\end{proof}

We now consider the consequences for bibranchings, which were introduced by Schrijver \cite{Schr82}. A directed graph is called partitionable if its vertex set $V$ can be partitioned into $V_{1}$ and $V_{2}$ such that there is no arc from $V_{2}$ to $V_{1}$. Let $D=(V,A)$ be such a digraph with partition $V_{1},V_{2}$, and let $\delta(V_{1},V_{2})$ denote the set of arcs from $V_{1}$ to $V_{2}$. A \textbf{$(V_{1},V_{2})$-bibranching} in $D$ is an arc set $F \subseteq A$ such that for every $v \in V_1$ there is a $v \to V_2$ path and for every $v \in V_2$ there is a $V_1 \to v$ path in $F$. Contrary to the case of matching forests and mixed edge covers, it can be decided in polynomial time if $E$ can be partitioned into $k$ bibranchings \cite{Schr82}. Given $k$  $(V_{1},V_{2})$-bibranchings $F_1,\dots,F_k$, let $N_i=F_i \cap \delta(V_{1},V_{2})$ and $B_i=F_i\setminus N_i$.
We now prove the following.

\begin{corollary}\label{cor:bb1}
Let $D=(V_1,V_2; A)$ be a partitionable digraph that can be partitioned into $k$ bibranchings. Then $D$ can be partitioned into  $k$ disjoint bibranchings $F_1,\dots,F_k$
such that, for every $i,j \in [k]$, we have $||F_i|-|F_j|| \leq 1$, $||N_i|-|N_j|| \leq 2$, and  $||B_i|-|B_j|| \leq 2$.
\end{corollary}

\begin{corollary}\label{cor:bb2}
Let $D=(V_1,V_2; A)$ be a partitionable digraph that can be partitioned into $k$ bibranchings. Then $D$ can be partitioned into  $k$ disjoint bibranchings $F_1,\dots,F_k$
such that, for every $i,j \in [k]$, we have $||F_i|-|F_j|| \leq 2$, $||N_i|-|N_j|| \leq 1$, and  $||B_i|-|B_j|| \leq 2$.
\end{corollary}

\begin{proof}
We construct a mixed graph $G$ from $D$ by replacing every arc in $\delta(V_{1},V_{2})$ by an edge, and reversing every arc in $E[V_1]$. If an arc set $F$ is a bibranching in $D$, then the corresponding set of edges and arcs in $G$ form a mixed edge cover, and vice versa. Thus, Theorems~\ref{thm:mec1} and \ref{thm:mec2} imply the corollaries.
\end{proof}

\section*{Acknowledgement}
The first author was supported by the Hungarian National Research, Development and Innovation Office -- NKFIH grant K120254, and by Thematic Excellence Programme, Industry
and Digitization Subprogramme, NRDI Office, 2019.
The second author was supported by JST CREST Grant Number JPMJCR14D2, JSPS KAKENHI Grant Number JP18K18004,
and MEXT Quantum Leap Flagship Program (MEXT Q-LEAP).


\end{document}